\documentclass[11pt]{article}

\usepackage{float}
\usepackage[colorlinks,linkcolor=blue]{hyperref}
\usepackage[latin1]{inputenc}
\usepackage{amsmath}
\usepackage{graphicx}
\usepackage{amssymb}
\usepackage{amsthm}
\usepackage{verbatim}
\usepackage{epsfig}
\usepackage[labelfont=bf]{caption}
\usepackage{enumerate}
\usepackage{subfig}
\usepackage{epstopdf}

\newtheorem{fact}{Fact}
\newtheorem{definition}{Definition}
\newtheorem{assumption}{Assumption}
\newtheorem{lemma}{Lemma}
\newtheorem{remark}{Remark}

\newtheorem{proposition}{Proposition}

\def\begquo{\begin{quote}}
\def\endquo{\end{quote}}
\def\begequarr{\begin{eqnarray}}
\def\endequarr{\end{eqnarray}}
\def\begequarrs{\begin{eqnarray*}}
\def\endequarrs{\end{eqnarray*}}
\def\begarr{\begin{array}}
\def\endarr{\end{array}}
\def\begequ{\begin{equation}}
\def\endequ{\end{equation}}
\def\lab{\label}
\def\begdes{\begin{description}}
\def\enddes{\end{description}}
\def\begenu{\begin{enumerate}}
\def\begite{\begin{itemize}}
\def\endite{\end{itemize}}
\def\endenu{\end{enumerate}}

\def\lef[{\left[\begin{array}}
\def\rig]{\end{array}\right]}
\def\qed{\hfill$\Box \Box \Box$}
\def\begcen{\begin{center}}
\def\endcen{\end{center}}
\def\begrem{\begin{remark}\rm}
\def\endrem{\end{remark}}
\def\begdef{\begin{definition}}
\def\enddef{\end{definition}}
\def\begpro{\begin{proposition}}
\def\endpro{\end{proposition}}
\def\begfac{\begin{fact}}
\def\endfac{\end{fact}}
\def\begass{\begin{assumption}}
\def\endass{\end{assumption}}

\def\begmat#1{\begin{bmatrix}#1\end{bmatrix}}
\def\begali#1{\begin{align}{#1}\end{align}}
\def\begalis#1{\begin{align*}{#1}\end{align*}}

\def\calg{{\cal G}}

\def\calh{{\cal H}}

\def\cals{{\cal S}}

\def\call{{\cal L}}

\def\cald{{\cal D}}
\def\caly{{\cal Y}}
\def\calw{{\cal W}}

\def\liminf{\lim_{t \to \infty}}
\def\litcallinf{\ell_\infty}
\def\callinf{{\cal L}_\infty}
\def\L2e{{\cal L}_{2e}}
\def\bul{\noindent $\bullet\;\;$}
\def\rea{\mathbb{R}}
\def\intnum{\mathbb{N}}

\def\diag{\mbox{diag}}
\def\adj{\mbox{adj}}
\def\col{\mbox{col}}

\def\et{\varepsilon}
\def\diag{\mbox{diag}}

\def\max{{\small{\mbox{max}}}}

\def\IJACSP{{\it Int. J. on Adaptive Control and Signal Processing}}

\def\TAC{{\it IEEE Trans. Automatic Control}}

\def\SCL{{\it Systems \& Control Letters}}

\def\calg{{\cal G}}

\usepackage{color}

\usepackage[prependcaption,colorinlistoftodos]{todonotes}

\graphicspath{{figs/}}

\begin{document}

\title{Parameter Estimation of Nonlinearly Parameterized Regressions without Overparameterization nor Persistent Excitation: Application to System Identification and Adaptive Control}


\author{ Romeo~Ortega$^{\dag,}$\thanks{R. Ortega is with Laboratoire des Signaux et Syst\`emes, CNRS--SUPELEC,
Gif--sur--Yvette, France, e-mail: ortega@lss.supelec.fr}, Vladislav Gromov\thanks{V. Gromov and A. Pyrkin are with Faculty of Control Systems and Robotics, ITMO University, Saint Petersburg, Russia, email: gromov$\{$pyrkin$\}$@itmo.ru}, Emmanuel Nu\~no\thanks{ E. Nu\~no is with Department of Computer Science, CUCEI, University of Guadalajara,  Guadalajara, Mexico, email: emmanuel.nuno@cucei.udg.mx}, Anton Pyrkin$^\dag$   and Jose~Guadalupe~Romero\thanks{J. G. Romero is with  Departamento Acad\'{e}mico de Sistemas Digitales, ITAM, Ciudad de M\'exico, M\'{e}xico, e-mail:jose.romerovelazquez@itam.mx}
}
\maketitle
%

\begin{abstract}
In this paper we propose a solution to the problem of parameter estimation of {\em nonlinearly parameterized regressions}---continuous or discrete time---and apply it for system identification and adaptive control. We restrict our attention to parameterizations that can be factorized as the product of two functions, a measurable one and a {\em nonlinear} function of the parameters to be estimated. Although in this case it is possible to define an {\em extended} vector of unknown parameters to get a linear regression, it is well-known that overparameterization suffers from some severe shortcomings. Another feature of the proposed estimator is that parameter convergence is ensured {\em without} a persistency of excitation assumption. It is assumed that, after a coordinate change, {\em some of the elements} of the transformed function satisfy a {\em monotonicity} condition. The proposed estimators are applied to design identifiers and adaptive controllers for nonlinearly parameterized systems. In continuous-time we consider a general class of nonlinear systems and those described by Euler-Lagrange models, while in discrete-time we apply the method to the challenging problems of direct and indirect adaptive pole-placement. The effectiveness of our approach is illustrated with several classical examples, which are traditionally tackled using overparameterization and assuming persistency of excitation. 
\end{abstract}
%
\section{Introduction and Literature Review}
\lab{sec1}
%
{I}{t} is well known that nonlinear parameterizations are inevitable in any realistic practical problem \cite{BASDOC,DASAND,IZH,LJU,NEL,ORTbook}. Unfortunately, designing adaptive (identification or control) algorithms for nonlinearly parameterized systems is a difficult poorly understood problem. Some results for {\em gradient estimators} have been reported in the literature for {\em convexly} parameterized continuous-time (CT) systems. It was first reported in \cite{FOMetal} (see also \cite{ort}) that convexity is enough to ensure that the gradient search
``goes in the right direction" in a {\em certain region} of the estimated parameter space. The idea is then to apply a standard adaptive scheme in this region,
while in the ``bad" region either the adaptation is frozen and a robust constant parameter controller is switched-on \cite{FRAetal} or, as
proposed in \cite{ANNetal}, the adaptation is running all the time and stability is ensured with a high-gain mechanism which is suitably adjusted
incorporating prior knowledge on the parameters. In \cite{NETetal} {\em reparametrization} to convexify an otherwise non-convexly parameterized
system is proposed. See also \cite{NETetal1} and \cite{TYUetal} for some interesting results along these lines, where the controller and the
estimator switch between over/underbounding convex/concave functions. 

On the other hand, using the Immersion and Invariance adaptation laws proposed in \cite{ASTKARORT}, stronger results were obtained in \cite{LIUetaltac,LIUetalscl} invoking the property of {\em monotonicity}, see also \cite{TYUetal,TYUetal1} for related results. The main advantage of using monotonicity, instead of convexity, is that in the former case the parameter search ``goes in the right direction" in all regions of the estimated parameter space---this is in contrast to the convexity-based designs where, as pointed out above, this only happens in some regions of this space.\footnote{The relation between these two approaches follows invoking Kachurovskii's Theorem that establishes the equivalence between convexity of a function and monotonicity of its gradient \cite[Theorem 4.1.4]{HIRLEM}, see also \cite{BOYVAN}.} 

Since this important difference is not always appreciated, let us illustrate it with the simple case of a scalar, CT, nonlinearly parameterized regression equation (NPRE)
$$
y=\calh(t,\theta),
$$
where $y(t) \in \rea$, $h:\rea_{>0} \times \rea^q \to \rea$ and $\theta \in \rea^q$ is the vector of unknown parameters. If we assume that $h$ is {\em convex} in $\theta$ the gradient descent search
$$
\dot {\hat \theta}=\left[{\partial \calh(t,\hat \theta) \over \partial \hat \theta}\right]^\top [\calh(t,\hat \theta)-y]
$$
ensures $\tilde \theta^\top  \dot{\tilde \theta} \leq 0$ {\em provided} $\calh(t,\hat \theta) \geq \calh(t, \theta)$, where we defined the parameter error vector $\tilde \theta:=\hat \theta- \theta$. On the other hand, if we assume that $h$ is {\em monotone} decreasing in $\theta$ the simple estimator
$$
\dot {\hat \theta}=\calh(t,\hat \theta)-y
$$
ensures $\tilde \theta^\top  \dot{\tilde \theta} \leq 0$ {\em all the time}.

To the best of the authors' knowledge no developments---similar to the ones mentioned above---have been reported for case of nonlinearly parameterized {\em discrete-time} (DT) regressions that, in spite of its great practical importance, have attracted less attention in the identification and adaptive control community. One of the objectives of our paper is to contribute, if modestly, towards the development of estimation algorithms for DT NPRE. In particular, we provide solutions to the, essentially open, problems of direct and indirect adaptive pole-placement control (APPC) without overparameterization nor persistency of excitation (PE) requirements.\footnote{We recall that a bounded vector signal $\Omega \in \rea^q$ is said to be PE if there exist $\delta \in \rea_{> 0}$ such that $\int_t^{t+T}\Omega(\tau)\Omega^\top(\tau)d\tau \geq \delta I_q$ for some $T \in \rea_{> 0}$ and all $t \in \rea_{\geq 0}$ in CT, or $\sum_{j=k+1}^{k + K} \Omega(j) \Omega^\top(j) \geq \delta I_m,\;\forall k \in \intnum_{\geq 0},$ for some $K \in \intnum_{> 0}$, with $K \geq m$, in DT.} It should be pointed out that a solution to the direct APPC problem using overparameterization, hence requiring some excitation conditions, has been recently reported in \cite{PYRetal}. 
  
A very important drawback of the aforementioned approaches is that the monotonicity or convexity conditions are imposed on functions that depend, not only on the parameters, but {\em also} on external signals, {\em e.g.}, time or the system state. This renders the verification of the condition very hard to carry out. This unfortunate situation happens even in the case when the uncertain terms appear as products of a function of the unknown parameters times a known function---the so-called, {\em factorizable mappings}, that is NPRE of the form
$$
y=\Omega \cals(\theta),
$$
with $\cals:\rea^q \to \rea^p$, with $p>q$.  Although in this case it is possible to define the extended parameter vector $\theta_a:=\cals(\theta)$ to obtain a linear parametrization, overparametrization suffers from the following well-known shortcomings \cite{IOASUN,LJU,SASBOD}.
\begenu[{\bf S1}]
\item Performance degradation, {\em e.g.}, slower convergence, due to the need of a search in a larger parameter space.
\item More stringent conditions imposed on the reference signals to ensure the PE requirement needed for convergence of the new parameters.
\item Inability to recover the true parameters---except for injecting mappings. This stymies the application of this approach in situations, where the actual parameters are needed, {\em e.g.}, in direct adaptive control.
\item Conservativeness introduced when incorporating prior knowledge in restricted parameter estimation.
\item Reduction of the domain of validity of the estimates stemming from the, in general only local, invertibility of the overparameterization mappings.
\endenu

In this paper we propose a parameter estimator for monotonic, factorizable NPRE that achieves the following objectives.

\begenu[{\bf O1}]
\item It {\em does not} rely on overparameterization.
\item Imposes the monotonicity property {\em directly} to the function $\cals(\theta)$.

\item Ensures parameter convergence {\em without} the stringent PE requirement.
\endenu

CT estimators for NPRE with factorizable mappings that avoid overparameterization and rely on monotonicity have been reported  in \cite[Section 3]{LIUetalscl} and \cite[Section III]{ARAetaltac}. In \cite{LIUetalscl} neither the second nor the third objectives above are achieved. On the other hand, in \cite{ARAetaltac} these objectives are achieved, via the use of a dynamic regressor extension and mixing (DREM) estimator. As is well known, the main feature of DREM is that it generates, out of a $q$-dimensional regression equations, one {\em scalar} equation for each of the $q$ unknown parameters. Another important feature of DREM is that parameter convergence is ensured without assuming PE. 

In this paper, we also use DREM to derive both, CT and DT, parameter estimators. We  obtain simpler and stronger results than \cite{ARAetaltac} due to the following three key modifications.
\begenu[{\bf M1}]
\item Generate the extended regressor matrix using the linear {\em time-varying} (LTV) operators first introduced in \cite{KRE}. This avoids the need to select several linear, scalar operators, whose choice is difficult to decide, and provides sharper convergence results. 
\item Directly apply the ``mixing" operation---that is the multiplication by the adjugate of the extended matrix---to generate the scalar regressions. This is in contrast to the unnecessarily complex matrix factorization proposed in  \cite{ARAetaltac}.     
\item Incorporate the possibility of adding a {\em change of coordinates} to the original parameters to satisfy the required monotonicity property.
\endenu
 
The remainder of the paper is organized as follows. First, we present in Section \ref{sec2} a general result of ``monotonizability" of factorizable NPRE. In Section \ref{sec3} we apply DREM to generate the scalar regressors. In Section \ref{sec4} the CT and DT DREM-based estimators are presented. Section \ref{sec5} is devoted to the application of this estimator to the problem of adaptive control of CT nonlinearly parameterized, nonlinear systems, with particular emphasis on Euler-Lagrange (EL) models. The case of DT NPRE is illustrated in Section \ref{sec6} with the example of identification of a solar heated house proposed in  \cite[pp. 130]{LJU}  and with the classical  problems of direct and indirect APPC \cite{GOOSIN}. The paper is wrapped-up with concluding remarks in  Section \ref{sec7}.\\

\noindent {\bf Notation.} $I_n$ is the $n \times n$ identity matrix. $\rea_{>0}$, $\rea_{\geq 0}$, $\intnum_{>0}$ and $\intnum_{\geq 0}$ denote the positive and non-negative real and integer numbers, respectively.   For $n \in \intnum_{>0}$ we define the set $\bar n :=\{1,2,\dots,n\}$. For $x \in \rea^n$, we denote $|x|^2:=x^\top x$. CT signals $s:\rea_{\geq 0} \to \rea$ are denoted $s(t)$, while for DT sequences $s:\intnum_{\geq 0} \to \rea$ we use $s(k):=s(kT_s)$, with $T_s \in \rea_{> 0}$ the sampling time. When a formula is applicable to CT signals and DT sequences the time argument is {\em omitted}. The action of an operator $H:\callinf \to \callinf$ on a CT signal $u(t)$ is denoted $H[u](t)$, while for an operator $H:\litcallinf \to \litcallinf$ and a sequence $u(k)$ we use  $H[u](k)$. With $i \in \intnum_{>0}$ we define the shift operator for DT sequences $q^{\pm i} u(k):=u(k \pm i)$, and the differentiation operator for CT signals $p^i[u](t):={d^i u\over dt^i}$. All mappings and reference signals are assumed {\em smooth}. Given a function $F:  \rea^n \to \rea$ we define the differential operators $\nabla F:=\left(\frac{\displaystyle \partial F}{\displaystyle \partial x}\right)^\top$ and $\nabla^2 F:=\frac{\displaystyle \partial^2 F}{\displaystyle \partial x^2}$. For general mappings $\cals:\rea^n \to \rea^n$, the $(i,j)$-th element of its Jacobian is defined as
$$
(\nabla \cals)_{(ij)}  := \frac{\partial \cals_j}{\partial x_i},\quad (i,j) \in \bar n \times \bar n.
$$
%
\section{Monotonic Nonlinearly Parameterized Factorizable Regressions}
\lab{sec2}
%
In this section we identify the class of NPRE that we consider in the paper. Namely, factorizable NPRE, where the mapping dependent on the unknown parameters verifies a {\em monotonicity} condition. 

\subsection{Problem formulation}
\lab{subsec21}
%
In many system identification and adaptive control applications one is confronted with the problem of estimation of the parameters appearing in a NPRE of the form
\begequ
\lab{nlpre}
y=\Omega \cals(\theta)+\et
\endequ 
where $y \in \rea^n,\;\Omega \in \rea^{n \times p}$ are {\em measurable} signals, $\theta \in \rea^q$ is a constant vector of {\em unknown} parameters, $\cals:\rea^q \to \rea^p$, with 
\begequ
\lab{pgreq}
p >q, 
\endequ
and $\et$ is a (generic) exponentially decaying term.  The task is to identify on-line the parameters $\theta$, out of the measurements of $y$ and  $\Omega$.

\begrem
\lab{rem1}
The NPRE \eqref{nlpre} is, of course, a particular case of the more general,  {\em non-factorizable}, regression $y(\cdot)=\calh(\cdot,\theta)$, with $(\cdot)=t$ in CT or $k$ in DT. But, it is more often encountered than the classical linear regression $y=\Omega\theta$---and the solution of the associated estimation problem is far more complicated. Although in the factorizable case it is possible to introduce extra parameters to obtain a linear parametrization, {\em e.g.}, define a {\em bigger} dimensional vector $\theta_a:=\cals(\theta) \in \rea^p$, overparametrization suffers from the well-known shortcomings {\bf S1}-{\bf S5} mentioned in the Introduction.
\endrem

\begrem
\lab{rem2}
For the sake of simplicity, we present $y$ and $\Omega$ as functions of time, in the understanding that they may be functions of measurable signals evaluated at time $t$ in CT or $k$ in DT, for instance, the state of a dynamical system---as shown below. Also, following standard practice, in the sequel we disregard the presence of the term $\et$, stemming from the effect of the initial conditions of various filters used to generate the regression, see \cite{ARAetaltac} for a discussion on this assumption.
\endrem

\subsection{Key monotonicity assumption}
\lab{subsec22}
%
Similarly to \cite{ARAetaltac,LIUetaltac,LIUetalscl} the key property of the parameterization that we will exploit is $P$-monotonicity, which is defined a follows.

\begin{definition}
\label{def1} \em
Given a {\em positive definite} matrix $P \in \rea^{q \times q}$, a mapping $\call: \rea^q \to \rea^q$ is {\em strongly $P$-monotone}  if and only if there exists a constant $\rho \in \rea_{>0}$ such that
\begequ
\lab{monpro}
(a-b)^\top P \left[\call(a) - \call(b)\right] \geq \rho|a-b|^2 >  0,\; \forall a,b \in \rea^q,\; a \neq b. 
\endequ
\end{definition}

The following interesting result of Demidovich \cite{DEM}---see also \cite{PAVetal}---provides a simple way to verify $P$-monotonicity.

\begin{lemma}
\lab{lem1} \em
A sufficient condition for a differentiable mapping $\call: \rea^q \to \rea^q$ to be strictly $P$--monotone is
\begequ
\lab{demcon}
P \nabla \call  + (\nabla \call)^\top P \geq \rho I_q > 0.
\endequ
\end{lemma}

The following ``monotonizability" assumption {\em via coordinate change} is instrumental for our further developments.

\begin{assumption}
\lab{ass1}\em
Consider the mapping $\cals(\theta)$. There exists: 
\begenu[{(i)}]
\item a bijective mapping $\cald:\rea^q \to \rea^q,\theta \mapsto \eta$ with right inverse $\cald^I:\rea^q \to \rea^q,\eta \mapsto \theta$;
\item a permutation matrix $T \in \rea^{p \times p}$ and;
\item a positive definite matrix $P \in \rea^{q \times q}$
\endenu
such that 
\begequ
\lab{demcon1}
P \nabla \calw(\eta) C^\top + C [\nabla \calw(\eta)]^\top P \geq \rho I_q >0,
\end{equation}
where
\begali{
\calw(\eta)& :=\cals(\cald^I(\eta)),\;
C :=\begmat{I_q & | & 0_{q \times (p-q)}}T.
\lab{psic}}
\end{assumption}
\qed

In words, the construction associated with Assumption \ref{ass1} proceeds as follows. First, introduce a bijective coordinate change for the parameters $\theta$, namely $\eta=\cald(\theta)$, with inverse $\theta=\cald^I(\eta)$. Second, write the original mapping $\cals(\theta)$ in terms of the parameters $\eta$ via the definition of the new mapping 
$$
\calw(\eta):=\cals(\cald^I(\eta)).
$$
Third, assuming that these mapping contains $q$ elements that are ``good"---term to be defined below---place them at the top with the permutation matrix $T$ and select them with the fat matrix $\begmat{I_q & | & 0_{q \times (p-q)}}$. Whence, define the new ``good" mapping $\calg:\rea^q \to \rea^q$ as
\begequ
\lab{gequcpsi}
\calg(\eta):=C \calw(\eta).
\endequ
Observing  that $\nabla \calg= \nabla \calw C^\top $, and invoking Lemma \ref{lem1}, the condition \eqref{demcon1} ensures that this ``good" mapping is strongly $P$-monotonic. For future reference we rewrite this condition in terms of the ``good" mapping as
\begequ
\lab{demcon2}
P \nabla \calg(\eta)+  [\nabla \calg(\eta)]^\top P \geq \rho I_q >0.
\end{equation}

Using the definitions above in the NPRE \eqref{nlpre} we obtain the new NPRE in terms of the parameters $\eta$ as
\begequ
\lab{newnpre}
y = \Omega \calw(\eta)
\endequ

\begrem
\lab{rem3}
To obtain a NPRE containing only ``good" functions---but without the essential parameter change $\theta \mapsto \eta$---a complicated reordering and mixing of the NPRE \eqref{newnpre} is proposed in \cite{ARAetaltac}. In the next subsection we show that direct application of DREM generates an alternative, much simpler procedure, to carry out this task.  
\endrem
%
\section{Generation of Scalar NPRE via DREM}
\lab{sec3}
%
In this section we apply DREM \cite{ARAetaltac}---with an LTV operator---to the $p$-dimensional NPRE  \eqref{newnpre}, and then select the ``good" terms via \eqref{gequcpsi} to generate $q$ {\em scalar} NPRE. First, we present the construction for CT signals and then treat the case of DT sequences. 
\subsection{Continuous-time case}
\lab{subsec31}
%
\begin{proposition}
\lab{pro1}\em
Consider the NPRE \eqref{newnpre} for CT signal. Define the signals
\begali{
\nonumber
\dot Y(t) &=-\lambda Y(t) + \Omega^\top(t) y(t)\\
\nonumber
\dot \Phi(t) &=-\lambda \Phi(t) + \Omega^\top(t)  \Omega(t)\\
\nonumber
\caly(t) &=C \adj\{\Phi(t)\}Y(t)\\
\lab{kfil}
\Delta(t) &=\det\{\Phi(t)\}.
}
The $q$ scalar NPRE
\begequ
\lab{scanpre}
\caly_i(t)=\Delta(t) \calg_i(\eta),\;i \in \bar q \;\Leftrightarrow\;\caly(t) = \Delta(t) \calg(\eta)
\endequ
hold.
\end{proposition}

\begin{proof}
Multiplying  \eqref{newnpre} by $\Omega^\top(t) $ and applying the stable, linear time-invariant (LTI) filter 
\begin{equation}
\label{fil}
H(p)={1 \over {p+\lambda}}
\end{equation}
where $\lambda \in \rea_{>0}$, we get
$$
H(p)[\Omega^\top y](t)=H(p)[\Omega^\top  \Omega](t) \calw(\eta),
$$
whose state realization is given in \eqref{kfil} and yields the relation
$$
Y(t)= \Phi(t) \calw(\eta). 
$$
Now, multiplying this equation by the {\em adjoint} of the extended regressor matrix $\Phi(t)$ we obtain
$$
\adj\{\Phi(t)\}Y(t)=\Delta(t) \calw(\eta)
$$
where we used the fact that for all---{\em possibly singular}---$p \times p$ matrices $A$ we have $\adj\{A\}A= \det\{A\} I_p.$ The proof is completed multiplying the last equation by $C$, invoking \eqref{gequcpsi}, and noting that $\Delta(t)$ is a {\em scalar}.

\end{proof}

\begrem
\lab{rem4}
The construction of the extended regressor $\Phi(t)$ proposed above is done following {\em verbatim} the DREM procedure of \cite{ARAetaltac} with LTV operators. This construction was first proposed in \cite{KRE} and is sometimes called Memory Regressor Extension \cite{GERetal}. 
\endrem

\subsection{Discrete-time case}
\lab{subsec32}
%
\begin{proposition}
\lab{pro2}\em
Consider the NPRE \eqref{newnpre} for DT sequences. Fix $0 < \alpha <1$ and define the signals
\begali{
\nonumber
Y(k) &=-\alpha Y(k-1) + \Omega^\top(k-1) y_p(k-1)\\
\nonumber
\Phi(k) &=-\alpha \Phi(k-1) + \Omega^\top(k-1)  \Omega(k-1)\\
\nonumber
\caly(k)  &=C \adj\{\Phi(k)\}Y(k)\\
\lab{kfildt}
\Delta(k) &=\det\{\Phi(k)\}.
}
The $q$ scalar NPRE
\begequ
\lab{scanpredt}
\caly_i(k)=\Delta(k) \calg_i(\eta),\;i \in \bar q \;\Leftrightarrow\;\caly(k)  = \Delta(k) \calg(\eta)
\endequ
hold.
\end{proposition}

\begin{proof}
The proof follows {\em verbatim} the one given in Proposition \ref{pro1} replacing the CT filter \eqref{fil} by the stable, LTI, DT filter ${1 \over {q+\alpha}}$.
\end{proof}

\begrem
\lab{rem5}
The construction of the extended regressor $\Phi(k)$ given above is the discrete-time version of the one proposed in \cite{KRE} and may be found in \cite{GERetal}. 
\endrem

\section{Parameter Estimators Convergence Analysis}
\lab{sec4}
%
In this section we present the CT and DT estimation laws for the parameters $\eta$ of the NPRE \eqref{scanpre} and \eqref{scanpredt}, respectively.
\subsection{Continuous-time case}
\lab{subsec41}
%
\begin{proposition}\em
\label{pro3}
Consider the NPRE \eqref{scanpre} satisfying \eqref{demcon2} of Assumption \ref{ass1}. Propose the parameter estimator
\begequ
\label{esteta}
\dot{\hat{\eta}}(t) = \Gamma P \Delta(t) [\caly(t) - \Delta(t) \calg(\hat \eta(t))],
\endequ
with $\Gamma \in \rea^{q \times q}$, $\Gamma >0$ the adaptation gain.
\begenu[(i)]
\item The {\em norm} of the parameter estimation vector $\tilde \eta(t):=\hat \eta(t) - \eta$ is monotonically non-increasing, that is,
\begequ
\lab{monproct}
|\tilde \eta(t_2) \leq |\tilde \eta(t_1)|, \forall t_2 \geq t_1 \in \rea_{\geq 0}. 
\endequ 
\item  The following implication  holds
$$
\Delta(t) \notin \call_2 \; \Rightarrow \; \lim_{t \to \infty}| \tilde \eta(t)| = 0.
$$ 
\endenu
\end{proposition}

\begin{proof}
Replacing \eqref{scanpre} in \eqref{esteta} we get the error equation
\[
\dot{\tilde{\eta}}(t) = -\Delta^2(t) \Gamma P [\calg(\hat\eta(t)) - \calg(\eta)].
\]
To analyse its stability define the Lyapunov function candidate $V(\tilde \eta) = \frac{1}{2} \tilde \eta^\top \Gamma^{-1} \tilde \eta$, whose derivative yields
\begalis{
\dot V(t)  & =  - \Delta^2(t) ( \hat \eta(t) - \eta)^\top P {[ \calg(\hat \eta(t)) - \calg(\eta)]} \\
& \leq  - \Delta^2(t) \rho | \tilde \eta(t)|^2  \\
& \leq  - \frac{2k}{\lambda_{\max}\{\Gamma\}}{\Delta}^2(t) V(t),
}
where we invoked Assumption \ref{ass1} to get the first bound, where $\lambda_{\max}\{\cdot\}$ denotes the maximum eigenvalue. The fact that $V(t)$ is non-increasing proves the first claim.

To prove the second one, we invoke the Comparison Lemma \cite[Lemma 3.4]{KHA} that yields the bound
$$
V(t) \leq e^{-\frac{2k}{\lambda_{\max}\{\Gamma\}} \int_0^t \Delta^2(s)ds}V(0),
$$ 
which ensures that $| \tilde \eta(t)| \to 0$ as $t \to \infty$ if  $\Delta(t) \notin \call_2$.
\end{proof}

\begrem
\lab{rem6}
As is well known, convergence in {\em all} parameter estimators---as well as in state observers---can only be ensured under some kind of excitation conditions \cite{LJU}. In particular, for standard gradient and least-squares estimators this property is encrypted in the well known PE requirement of the regressor \cite{GOOSIN,IOASUN,SASBOD}. As it has been shown in \cite{ARAetaltac} convergence of DREM estimators can be ensured without requiring PE and replacing it, instead, by the assumption $\Delta(t) \notin \call_2$, which is necessary and sufficient for parameter convergence for {\em linear} regression equation. As shown in Proposition \ref{pro3} this condition is sufficient for NPRE of the form  \eqref{nlpre} with a $P$-``monotonizable" regressor $\cals(\theta)$. Notice, on the other hand, that the nice property of element-by-element monotonocity of the parameter estimation errors of linear regressions is lost, and we can only ensure that the norm of this vector is monotonically non-increasing. 
\endrem

\begrem
\lab{rem7}
It is interesting to note that the following important implication for the CT DREM given above was recently proven in \cite{KORetal}: 
$$
\Omega(t) \in PE\;\Rightarrow\;\Delta(t) \in PE.
$$ 
Hence, if the standard gradient estimator for the {\em overparameterized} linear regression $y(t)=\Omega(t) \eta_a$, with $\eta_a:=\calw(\eta)$ is globally exponentially stable (GES) also the DREM estimator is GES. However, asymptotic convergence of DREM is ensured with the condition $\Delta(t) \notin \call_2$, which is {\em strictly weaker} than $\Delta(t) \in PE$.
\endrem

\begrem
\lab{rem8}
We have assumed that the mapping $\calg(\eta)$ is {\em strongly} $P$-monotonic. Its clear from the derivations above that this requirement can be relaxed to {\em strictly} $P$-monotonic adding some further assumptions on $\Delta(t)$. 
\endrem
\subsection{Discrete-time case}
\lab{subsec42}
%
In this subsection we present the estimation law for the parameters $\eta$ of the DT NPRE \eqref{scanpredt}. Towards this end, the following is needed. 
\begin{assumption}
\label{ass2}\em
The mapping $\calg(\eta)$ satisfies the {\em Lipschitz condition}
\begequ
\lab{lipcon}
|\calg(a) - \calg(b)| \leq \nu |a-b|,\; \forall a,b \in \rea^q,
\endequ
for some $\nu>0$.
\end{assumption}

The proposition below presents four different stability properties of the proposed DREM estimator.
 
\begin{proposition}\em
\label{pro4}
Consider the DT NPRE \eqref{scanpredt} with $\calg(\eta)$ satisfying \eqref{demcon2} of Assumption \ref{ass1} and Assumption \ref{ass2}. Propose the DT parameter estimator
\begequ
\label{dtesteta}
\hat{\eta}(k+1) = \hat \eta(k) + \gamma P {\Delta(k) \over 1 + \kappa\Delta^2(k)} [\caly(k)  - \Delta(k) \calg(\hat \eta)],
\endequ
with $\gamma>0$ the adaptation gain selected such that the constant\footnote{Clearly, for any positive $\rho,\nu$ and $P>0$, this condition is satisfied with  $\gamma < {2\rho \over  \nu^2 \lambda^2_{\max}\{P\}}$.}
\begequ
\lab{sig}
\sigma:= 2 \gamma \rho -{\gamma^2 \nu^2} \lambda^2_{\max}\{P\} >0,
\endequ
and the constant $\kappa$ verifying
\begequ
\lab{gaicon}
\kappa \geq \max\{1,\sigma\}.
\endequ
\begenu[{\bf P1}]
\item The {\em norm} of the parameter estimation error $\tilde \eta(k):=\hat \eta(k) - \eta$ is monotonically non-increasing, that is,
\begequ
\lab{monprodt}
|\tilde \eta(k_2) \leq |\tilde \eta(k_1)|, \forall k_2 \geq k_1 \in \intnum_{\geq 0}. 
\endequ 
\item The following implication is true
$$
\prod_{i=0}^\infty { 1 + \left(\kappa - \sigma\right) \Delta^2(k) \over 1 + \kappa \Delta^2(k)}=0 \; \Rightarrow \;\lim_{k \to \infty}|\tilde \eta(k)| = 0.
$$
\item The following implication is true
$$
\lim_{k \to \infty}\Delta(k)=: \Delta(\infty)\neq 0\;\Rightarrow\;\lim_{k \to \infty}|\tilde \eta(k)| = 0.
$$
\item Assume 
\begequ
\lab{connu}
\nu \leq {\rho \over \lambda_{\max}\{P\}},
\endequ
and pick $\gamma$ in the interval
\begequ
\lab{gamminmax}
{1 \over \nu^2 \lambda^2_{\max}\{P\}}\left[\rho -\sqrt{\rho^2- \nu^2 \lambda^2_{\max}\{P\}}\right] \leq \gamma \leq {1 \over \nu^2 \lambda^2_{\max}\{P\}}\left[\rho +\sqrt{\rho^2- \nu^2 \lambda^2_{\max}\{P\}}\right].
\endequ
The following  holds
$$
\Delta(k) \notin \ell_2 \; \Rightarrow \;\lim_{k \to \infty}|\tilde \eta(k)| = 0.
$$
\endenu
\end{proposition}

\begin{proof}
First, we observe that the condition \eqref{gaicon}, on one hand,  ensures the following bound for the  normalized scalar regressor
\begequ
\lab{bounorreg}
\bar \Delta^2(k):= {\Delta^2(k) \over 1 + \kappa \Delta^2(k)} \leq 1,
\endequ
and, on the other hand, shows that
\begequ
\lab{kapminsig}
\kappa - \sigma \geq 0.
\endequ
Replacing \eqref{scanpredt} in \eqref{dtesteta} we get the error equation
\[
{\tilde{\eta}}(k+1) ={\tilde{\eta}}(k) - \gamma P{\bar \Delta^2(k)}  [\calg(\hat\eta(k)) - \calg(\eta)],
\]
where we invoked the definition in \eqref{bounorreg}. To analyze the stability of this equation define the Lyapunov function candidate 
\begequ
\lab{lyafun}
V(k) = \frac{1}{2 \gamma} |\tilde \eta(k)|^2,
\endequ
which satisfies
\begali{
\nonumber
V(k+1)  & = V(k) - \bar \Delta^2(k) \tilde \eta^\top(k) P [ \calg(\hat\eta(k)) - \calg(\eta)] \\
\nonumber
&+ {\gamma \over 2}  \bar \Delta^4(k) [ \calg(\hat\eta(k)) - \calg(\eta)]^\top  P^2 [\calg(\hat\eta(k)) - \calg(\eta)] \\
\nonumber
& \leq V(k) - \rho  \bar \Delta^2(k) |\tilde \eta(k)|^2 + {\gamma \nu^2\over 2}\lambda^2_{\max}\{P\}   \bar \Delta^4(k) |\tilde \eta(k)|^2 \\ 
\nonumber
& \leq V(k) - \rho  \bar \Delta^2(k) |\tilde \eta(k)|^2 + {\gamma \nu^2\over 2}\lambda^2_{\max}\{P\}   \bar \Delta^2(k) |\tilde \eta(k)|^2 \\ 
\nonumber
& = V(k) -   \left[\rho -{\gamma \nu^2\over 2} \lambda^2_{\max}\{P\}\right]\bar \Delta^2(k)|\tilde \eta(k)|^2 \\ 
\nonumber
& =\left[1-  \sigma \bar \Delta^2(k)\right]V(k)  \\
\lab{vk}
& ={1+ \left(\kappa - \sigma\right) \Delta^2(k) \over  1  + \kappa \Delta^2(k)} V(k) ,
}
where we invoked Assumption \ref{ass1} and \eqref{lipcon}  to get the first bound, inequality \eqref{bounorreg} for the second bound, \eqref{sig} and \eqref{lyafun} in the third identity and the definition of $\bar \Delta(k)$ for the last identity.\\ 

\noindent [Proof of Property {\bf P1}] The proof is completed observing that \eqref{kapminsig} ensures
$$
{1+ \left(\kappa - \sigma\right) \Delta^2(k) \over  1  + \kappa \Delta^2(k)}\geq 0,
$$
consequently $V(k)$ is a non-increasing sequence.\\

\noindent [Proof of Property {\bf P2}] From 
$$
V(k+1) \leq \prod_{i=0}^k { 1 + \left(\kappa - \sigma\right) \Delta^2(i)  \over 1 + \kappa \Delta^2(i)}V(0),
$$
the claim follows immediately.\\

\noindent [Proof of Property {\bf P3}] To prove the second claim we first notice that, from the second inequality in \eqref{vk} and  \eqref{sig}, we get the bound
\begalis{
V(k+1)  
& \leq V(k) -  {\sigma \over 2 \gamma} \bar \Delta^2(k) |\tilde \eta(k)|^2 
}
summing the inequality above we get
$$
V(k)-V(0) \leq - \sum_{j=1}^k {\sigma \over 2 \gamma} \bar \Delta^2(j)|\tilde \eta(j)|^2\;\Rightarrow\; {2 \gamma V(0) \over \sigma} \geq  \sum_{j=1}^k \bar \Delta^2(j)|\tilde \eta(j)|^2.
$$
Taking the limit as $k \to \infty$ in the right hand side inequality we conclude that 
\begequ
\lab{limpro}
\bar \Delta(k)|\tilde \eta(k)| \in \ell_2  \; \Rightarrow \;\bar \Delta(k)|\tilde \eta(k)| \to 0,
\endequ
independently of the behaviour of $\Delta(k)$. Now, from the Algebraic Limit Theorem \cite[Theorem 3.3]{RUD} we know that the limit of the product of two convergent sequences is the product of their limits. On the other hand, from the fact that
$$
V(k+1) \leq V(k) \leq V(0),\;\forall k \in \intnum_{>0},
$$ 
we have that $|\tilde \eta(k)|$ is a bounded monotonic sequence, hence it converges \cite[Theorem 3.14]{RUD}. Finally, if $\Delta(k)$ converges to a {\em non-zero} limit, $\bar \Delta(k)$ also converges to a non-zero limit and we conclude from \eqref{limpro} that  $|\tilde \eta(k)| \to 0$. \\

\noindent [Proof of Property {\bf P4}] To prove the third claim we first observe that the condition \eqref{connu} ensures that the upper and lower limits of $\gamma$ given in \eqref{gamminmax} are well defined. Some lengthy, but straightforward calculations, show then that the conditions \eqref{connu} and \eqref{gamminmax} guarantee that $\sigma \geq 1$. Hence, in view of \eqref{gaicon}, the bound \eqref{bounorreg} as well as the derivations in \eqref{vk}, still hold. Then, setting $\kappa=\sigma$ in the last equation of \eqref{vk} we get
$$
V(k+1) \leq \prod_{i=0}^k {1 \over 1 + \kappa \Delta^2(i)}V(0)
$$
The proof is completed recalling that
$$
\prod_{i=0}^\infty {1 \over 1 + \kappa \Delta^2(i)}=0\;\Leftrightarrow\; \Delta(k) \notin \ell_2.
$$
\end{proof}

\begrem
\lab{rem9}
Similarly to the observation made in Remark \ref{rem6} the sufficient conditions for parameter convergence of Properties {\bf P2}-{\bf P4} should be interpreted as {\em excitation requirements} imposed on $\Delta(k)$. 
Notice that the condition of Property {\bf P3} is {\em sufficient} to ensure $\Delta(k) \notin \ell_2$ and {\em necessary} for it to be PE. In Property {\bf P4} we prove that $\Delta(k) \notin \ell_2$ is sufficient for parameter convergence but, unfortunately, we need to impose the rather ``unnatural" condition \eqref{connu}. Indeed, roughly speaking, the Lipschitz constant $\nu$ is related with an ``upper bound" on the derivative of $\calg(\eta)$ \cite[Theorem 9.19]{RUD}, while at the same time a high monotonicity degree $\rho$ requires this derivative to be large---which is in contradiction with  \eqref{connu}.
\endrem

%
\section{Application to CT Nonlinearly Parameterized Nonlinear Systems}
\lab{sec5}
%
In the section we apply the results on parameter estimation of CT NPRE of the previous section to tackle the problem of adaptive control of uncertain, nonlinearly parameterized, nonlinear systems. First, we treat the case of a rather general class of systems, then we specialize the result for EL models.
\subsection{Direct adaptive control of a general class of CT nonlinear systems}
\lab{subsec51}
%
Consider CT systems described by the state equations
\begin{equation}
\label{bassys}
\begin{array}{rcl}
    \dot{x}(t) & = & F(x(t) ,u(t) ) + R(x(t) )\cals(\theta)
\end{array}
\end{equation}
where $ x(t) \in \mathbb{R}^{n}$ is the measurable state, $u(t) \in \rea^m$, with $n \geq m$, is the control signal, the mappings $F: \rea^n\times \rea^m \to \rea^n$, $R: \rea^n\to \rea^{n \times p}$ and  $\cals: \rea^q\to \rea^p$ are {\em known} with $p>q$, and $\theta\in \mathbb{R}^q$ is a constant vector of {\em unknown parameters}. 

To streamline the formulation of the adaptive control problem we require the following {\em sine qua non} stabilizability condition. 

\begin{assumption}
\label{ass3}\em
There exists a mapping $\beta: \rea^n \times \rea^q \to \rea^m$, such that the system
\begin{equation}
\label{f-star}
    \dot{x}(t) = F(x(t) ,\beta(x(t) ,\theta)) + R(x(t))\cals(\theta)=:f_\star (x(t) )
\end{equation}
has a {\em globally exponentially stable} (GES) equilibrium at a desired value $x_\star \in \rea^n$.
\end{assumption}

The {\em control objective} is then to design a parameter estimator such that the (certainty-equivalent) adaptive control $u=\beta(x(t) ,\hat \theta(t) )$ ensures the asymptotic convergence 
\begequ
\lab{asycon}
\liminf x(t) = x_\star,
\endequ
with all signals bounded. To solve this problem we will impose Assumption \ref{ass1} to the mapping $\cals(\theta)$ and apply the estimator of Proposition \ref{pro3} to generate the adaptive controller.

A fist step in the design is the derivation of the NPRE \eqref{nlpre} for the system \eqref{bassys}. This is easily obtained  applying to \eqref{bassys} the stable, LTI filter \eqref{fil}  and defining
\begali{
\nonumber
y(t)  & :=p H(p)[x](t)  - H(p)[F(x,u)](t) \\
\Omega &:=H(p)[R(x)](t) ,
\lab{fillinparpla}
}
and $\et(t) $ is the solution of $H(p)[\et](t)  =0$. A state-space realization of \eqref{fillinparpla} is given by
\begali{
\nonumber
\dot z(t)  & = - \lambda(z(t) +x(t) )-F(x(t) ,u(t) )\\
\nonumber
\dot \Omega(t) &= - \lambda \Omega(t)  + R(x(t) )\\
y(t)  &= z(t) +x(t) .
\lab{stareapar}
}
We are in position to state the main result of this subsection.

\begin{proposition}\em
\label{pro5}
Consider the nonlinearly parameterized, nonlinear system \eqref{bassys} satisfying Assumptions \ref{ass1} and \ref{ass3}. Let the adaptive control be given by
$$
u(t) =\beta(x(t) ,\cald^I(\hat \eta(t) )),
$$
together with the parameter estimator \eqref{kfil}, \eqref{esteta} and \eqref{stareapar}. If $\Delta(t)  \notin \call_2$ we have that  \eqref{asycon} holds with all signals bounded.
\end{proposition}

\begin{proof}
First, notice that the closed-loop system takes the form
\begalis{
 \dot x(t)  & =  F(x(t) ,\beta(x(t) ,\hat \theta(t) )) +  R(x(t) )\cals(\theta)\\
        &  =  f_\star (x(t) ) + \xi(x(t),\tilde \theta(t)),
}
where we defined the perturbation term
$$
\xi(x(t),\tilde \theta(t)):=F(x(t) ,\beta(x(t) ,\tilde \theta(t)+\theta) -  F(x(t) ,\beta(x(t) ,\theta)).
$$
Using the fact that $\tilde \theta(t)=\cald^I(\tilde \eta(t))$ we see that the closed-loop system takes a cascade form
\begali{
\nonumber 
\dot x(t) &  =  f_\star (x(t) ) + \xi(x(t),\cald^I(\tilde \eta(t)))\\
\lab{cassys} 
\dot{\tilde{\eta}}(t) &= -\Delta^2(t) \Gamma P [\calg(\tilde\eta(t)+\eta) - \calg(\eta)],
}
with $ \xi(x(t),0)=0$. Assumption \ref{ass3} ensures that $x_\star$ is a GES equilibrium of the unperturbed system. Therefore, by \cite[Lemma 4.6]{KHA} the perturbed system is ISS with respect to the input $\tilde \eta(t)$. Now, the condition $\Delta \notin \call_2$ ensures that the origin of the $\tilde \eta(t)$ subsystem is globally asymptotically stable (GAS). Hence, by \cite[Lemma 4.7]{KHA},  the cascaded system \eqref{cassys} is GAS and, consequently,  \eqref{asycon} holds with all signals bounded.
\end{proof}

\begrem
\lab{rem10}
To simplify the presentation we have restricted ourselves to {\em regulation} tasks with static state-feedback controllers and aimed at {\em global} properties. The extension for tracking with dynamic controllers and local results follows {\em verbatim}. In particular, {\em local} asymptotic stability follows replacing GES by GAS in Assumption \ref{ass3}.  
\endrem
\subsection{Adaptive Control of Euler-Lagrange Systems}
\lab{subsec52}
%
In this subsection we specialize the result of the previous subsection to the practically important case of CT EL systems. On the other hand, we extend the scenario to treat the problem of {\em tracking} a reference for the state vector. To simplify the notation, throughout this section we omit the time dependence from all signals.

\subsubsection{System dynamics and adaptive control problem formulation}
\lab{subsubsec521}
%
We consider $n_q$ degrees-of-freedom (dof), possibly underactuated, EL systems with generalized coordinates $q(t)\in \rea^{n_q}$ and control vector $u(t) \in \rea^m$, $m \leq n_q$, whose dynamics is described by the EL equations of motion
\begin{equation}
\label{EL1}
{ \frac{d}{dt} \left[\nabla_{\dot q}\mathbb{L}(q,\dot q)\right] - \nabla_q \mathbb{L}(q,\dot q)= G(q) u,}
\end{equation}
where $\mathbb{L}:\rea^{n_q} \times \rea^{n_q} \to \rea$ is the Lagrangian function defined as
$$
\mathbb{L}(q,\dot q) :=   \mathbb{T}(q,\dot q) - \mathbb{U}(q),
$$
with  $\mathbb{T}:\rea^{n_q} \times \rea^{n_q} \to \rea$ the kinetic co-energy function and $\mathbb{U}:\rea^{n_q} \to \rea$ the potential energy function and $G:\rea^{n_q} \to \rea^{n_q \times m}$ is the full-rank input matrix. We restrict our attention to {\em simple} EL systems, whose kinetic energy is of the form
$$
\mathbb{L}(q,\dot q)=\frac{1}{2}\dot q^{\top}M(q)\dot q,
$$
where $M:\rea^{n_q} \to \rea^{n_q \times n_q}$is the generalized inertia matrix, which is positive definite and assumed to be {\em bounded}. See \cite{ORTbook} for additional details on this model and many practical examples.

For future reference we find convenient to write the dynamics of the EL system \eqref{EL1} as
\begin{equation}
\label{EL2}
{d \over dt }\left[  M(q) \dot q \right] - \frac{1}{2} \nabla_q \left[  \dot q^{\top}M(q)\dot q\right]   + \nabla \mathbb{U}(q) = G(q)u
\end{equation}
with the more explicit form 
\begin{equation}
\label{robdyn}
M(q) \ddot q+ C(q, \dot q)\dot q + \nabla{\mathbb{U}}(q) =G(q) u,
\end{equation}
where $C: \rea^{n_q} \times \rea^{n_q} \to \rea^{n_q \times n_q} $ represents the Coriolis and centrifugal forces matrix. As is well known \cite[Lemma 2.8]{ORTbook}, if the matrix $C(q,\dot q)$ is defined via the Christoffel symbols of the first kind, the key skew-symmetry property
\begequ
\lab{skesym}
z^\top[\dot M(q)-2C(q,\dot q)]z,\;\forall z \in \rea^{n_q},
\endequ
holds.  

Similarly to the previous subsection, we require the existence of a global {\em tracking} controller. 

\begin{assumption}
\label{ass4}\em
Given a desired bounded trajectory for the state vector $(q_\star(t), \dot q_\star(t)) \in \rea^{n_q}\times \rea^{n_q}$. Define the state tracking error $\col(\tilde q,\dot {\tilde q}):=\col(q-q_\star,\dot q - \dot q_\star).$ There exists a mapping $\beta: \rea^{n_q} \times \rea^{n_q} \times \rea^q \times \rea_{\ge 0} \to \rea^m$, such that the system
$$
M(q) \ddot q+ C(q, \dot q)\dot q + \nabla {\mathbb{U}}(q) =G(q) \beta(q,\dot q,\theta,t),
$$
has an {\em error dynamics} 
$$
\begmat{\dot {\tilde q} \\ \ddot {\tilde q}}=f_\star(\tilde q, \dot {\tilde q},t)
$$
whose origin is GES. 
\end{assumption}

The {\em control objective} is then to design a parameter estimator such that the (certainty-equivalent) adaptive control $u=\beta(q,\dot q,\hat \theta,t)$ ensures global asymptotic tracking, that is, 
\begequ
\lab{asyconel}
\liminf \col(\tilde q(t),\dot {\tilde q}(t)) = 0,
\endequ
with all signals bounded. 
\subsubsection{Derivation of the regression equation}
\lab{subsubsec522}
%
A fist step in the design is the derivation of the NPRE \eqref{nlpre} for the system \eqref{robdyn}---which was already reported in \cite{SLOLIaut}. Towards this end, we introduce the following parameterization of the inertia matrix $M(q)$ and the potential energy ${\mathbb{U}}(q)$ 
\begin{equation}
\label{parmu}
M(q) = \sum_{i=1}^\ell   m_i(q) \cals_i^m(\theta), \hspace{.5cm} {\mathbb{U}}(q)=\sum_{j=1}^r {\mathbb{U}}_j(q)  \cals_j^{\mathbb{U}}(\theta)
\end{equation}
 with  {\it known} matrices  $m_i: \rea^{n_q} \rightarrow  \rea^{n_q\times n_q}$ and functions ${\mathbb{U}}_j: \rea^{n_q} \rightarrow \rea$ and known  functions $\cals_i^m(\theta), \cals_j^{\mathbb{U}}(\theta):\rea^q \to \rea$ of the {\em unknown} physical parameters $\theta \in \rea^q$. We group together all functions $\cals_i^m(\theta), \cals_j^{\mathbb{U}}(\theta)$ in a single vector mapping $\cals:\rea^q \to \rea^p$ as
\begequ
\lab{thesthe}
\cals(\theta):=\col(\cals^m_1(\theta), \cdots, \cals^m_\ell(\theta), \cals^{\mathbb{U}}_{1}(\theta),  \cdots, \cals^{\mathbb{U}}_{r}(\theta) ) \in \rea^p,
\endequ
where $p:=\ell+r > q$. We are in position to present the following.
\begin{proposition} \em
\label{pro6}
There exists a  regressor matrix   $\Omega: \rea^{n_q} \times \rea^{n_q} \to \rea^{n_q \times p}$ such that the EL system \eqref{robdyn} satisfies the NPRE 
\begin{equation}
\label{MSNPRE}
y= \Omega  (q,\dot q) \cals(\theta)
\end{equation}
where 
\begequ
\lab{yel}
y:=H(p)\left[G(q) u \right],
\endequ
with $\theta$ and $\cals(\theta)$ defined via \eqref{parmu} and \eqref{thesthe}.
\end{proposition}
%
\begin{proof}

Applying the LTI filter \eqref{fil} to both sides of \eqref{EL2} we get
\begin{equation}
\label{newy}
p H(p)[ M(q) \dot q ]- \frac{1}{2}  H(p)\left[ \nabla_q (  \dot q^{\top}M(q)\dot q )\right] + H(p)[ \nabla {\mathbb{U}}(q)] = y,
\end{equation}
where we have used \eqref{yel}. 

Now, using the parameterization \eqref{parmu},  the left hand side of  \eqref{newy}  can be written as 
 \begin{equation}
\sum_{i=1}^\ell  H(p)  \Big [  p  m_i(q)  \dot q-  \frac{1}{2} \nabla_q (  \dot q^\top  m_i(q)  \dot q  )\Big ]\cals_i^m(\theta) +  \sum_{j=1}^r H(p)[  \nabla {\mathbb{U}}_j(q)]  \cals_j^{\mathbb{U}}(\theta) = \Omega (q, \dot q) \cals(\theta) 
  \label{Ec1}
 \end{equation}  
where we used \eqref{thesthe} and defined the regressor matrix
\begali{
\lab{omeel}
\Omega(q,\dot q):= H(p) \begmat{p  m_1(q)  \dot q -{1\over 2} \nabla_q  (\dot q^\top  m_1(q) \dot q )\\\vdots \\p m_\ell(q)-{1\over2} \nabla_q (\dot q^\top m_\ell(q) \dot q)\\ \nabla {\mathbb{U}}_1(q)\\ \vdots \\\nabla {\mathbb{U}}_r(q)}^\top,
}
this completes the proof.
\end{proof}

\begrem
\lab{rem11}
Notice that the terms $ H(p)[p  m_i(q)  \dot q],\;i \in \bar \ell$, may be written as ${p \over p+\lambda}[m_i(q)\dot q]$, hence they can be computed without differentiation.
\endrem

\begrem
\lab{rem12}
In \cite{SLOLIaut} an alternative parameterization of the EL system \eqref{EL1} is proposed. Indeed, applying the filter \eqref{fil} to the well-known power-balance equation \cite[Proposition 2.5]{ORTbook}
$$
\dot {\mathbb{H}}=\dot q^\top G(q)u,
$$
where $\mathbb{H}(q,\dot q):=\mathbb{T}(q,\dot q)+{\mathbb{U}}(q)$ is the total energy function, it is possible to obtain a NPRE of the form \eqref{MSNPRE} with {\em scalar} $y$ and $\Omega: \rea^{n_q} \times \rea^{n_q} \to \rea^{p}$. As argued in  \cite{SLOLIaut} this is a much {\em simpler} parameterization than the one given in Proposition \ref{pro6}.  However, extensive simulated evidence shows that this yields a {\em non-identifiable} parameterization. 
\endrem
\subsubsection{Main stabilization result}
\lab{subsubsec523}
%
We are now in position of present the main result of this subsection, whose proof follows {\em verbatim} the proof of Proposition \ref{pro5}, therefore it is omitted.

\begin{proposition} \em
\lab{pro7}
Consider the EL system \eqref{robdyn} with NPRE \eqref{MSNPRE} verifying Assumptions \ref{ass1} and \ref{ass4}. Let the adaptive control be given by
 \begin{equation}
 \label{con1}
 u=\beta(q,\dot q, \cald^I (\hat \eta),t)  
 \end{equation}
together with the parameter estimator \eqref{kfil}, \eqref{esteta}, \eqref{yel} and \eqref{omeel}. If $\Delta \notin \call_2$ we have that  \eqref{asyconel} holds with all signals bounded.
\end{proposition} 

In what follows we present two well-known choices of $ \beta(q,\dot q,\theta,t)$ for {\em fully actuated} systems, {\em i.e.}, $m=n_q$, and prove that they satisfy the key GES Assumption \ref{ass4} 

\bul The {\em Computed Torque Controller} in the known parameter case is given by
$$
\beta(q,\dot q,\theta,t)= M(q)[ \ddot q_\star  - K_1 \dot {\tilde q} - K_2 \tilde q] + C(q,\dot q) \dot q + g(q),
$$
resulting in the LTI closed-loop system
$$
\ddot {\tilde q} + K_1 \dot {\tilde q} + K_2 \tilde q=0,
$$
that, obviously, has a GES equilibrium at the origin for all positive definite control gains $K_1, K_2 \in \rea^{n_q \times n_q}$.

\bul The {\em Slotine-Li Controller} in the known parameter case is given by \cite{SLOLItac}
\begequ
\lab{slolicon}
\beta(q,\dot q,\theta,t)= M(q) \ddot q_r + C(q,\dot q) \dot q_r + g(q) +K_1 s,
\endequ
where we defined the signals
\begin{equation}
	\label{SLerrors}
\dot q_r := \dot q_\star  - K_2 \tilde q,\quad s := \dot {\tilde q} + K_2 \tilde q.
\end{equation}
The closed-loop system is then
$$
	M(q)\dot s +[ C(q,\dot q) + K_1] s = 0,\;\dot {\tilde q}+ K_2 \tilde q=   s,
$$
that---as indicated in \cite[Remark 4.5]{ORTbook}, see also \cite{SPOetal}---has an GES equilibrium at the origin. 

\begrem
\lab{rem13}
To the best of our knowledge, the proof of global stability of the adaptive version of the computed torque scheme proposed above is the first one reported in the literature.
\endrem
\subsubsection{Verifying Assumption \ref{ass1} on a $2$-DOF robot manipulator}
\lab{subsubsec524}
In this subsubection we show that the ``monotonizability" Assumption \ref{ass1} is verified for a 2-dof robot manipulator. The equation of motion of the robot is given by \eqref{robdyn} with  
\begali{
\nonumber
M(q)&=\begmat{ \cals_1(\theta) +2 \cals_2(\theta) \cos(q_2) & \cals_3(\theta) +\cals_2(\theta) \cos(q_2)  \\ \cals_3(\theta) +\cals_2(\theta)\cos(q_2)   &  \cals_3(\theta)}\\
      \label{MU}  
      {\mathbb{U}}(q)& =\cals_4(\theta) g \left(1+\sin(q_1+q_2)\right) + \cals_5(\theta) g \left(1+\sin(q_1)\right),
}
with $g$ the gravitational constant, the physical parameters $\theta:=\col(l_1,l_2,m_1,m_2)$, where $l_i>0$ is the the length of the link $i$ with mass $m_i>0$ for  $i=1,2$, and the mappings
\begequ
\lab{Sst}
\cals^m(\theta) := \begmat{ \theta^2_2 \theta_4+ \theta_1^2 (\theta_3 + \theta_4)  \\ \theta_1 \theta_2 \theta_4 \\ \theta_2^2 \theta_4 },\; S^{\mathbb{U}}(\theta):=\begmat{\theta_2 \theta_4 \\ \theta_1 (\theta_3 + \theta_4)},\; \cals(\theta):= \begmat{\cals^m(\theta)\\ S^{\mathbb{U}} (\theta))}.
\endequ

In the following lemma we verify Assumption \ref{ass1} for the mapping $\cals(\theta)$.

\begin{lemma}
\lab{lem2}\em
Consider the vector $\theta \in \rea^4_{>0}$ and the mapping $\cals:\rea^4_{>0} \to \rea^5_{>0}$  given by \eqref{Sst}. Assume the bounds
\begequ
\lab{bouthe}
\theta_1 \leq \theta^{M}_1,\;\theta_2^{m} \leq \theta_2 \leq \theta_2^{M},\; \theta_4^{m} \leq \theta_4.
\endequ
The mapping $\cald:\rea^4_{>0} \to \rea^4_{>0}$
$$
\eta=\cald(\theta)=\col( \theta_1, \theta_2, \theta_2 \theta_4, \theta_1 (\theta_3 + \theta_4)),
$$
with right inverse $\cald^I:\rea^4_{>0} \to \rea^4$
\begin{equation}
	\label{DI2gdl}
\theta=\cald^I(\eta)=\col(\eta_1, \eta_2, {\eta_4 \over \eta_1}- {\eta_3 \over \eta_2},{\eta_3 \over \eta_2}),	
\end{equation}
verifies Assumption \ref{ass1} with
$$
T =\begmat{0 & 1 & 0 & 0 & 0\\ 0 & 0 & 1 & 0 & 0\\0 & 0 & 0 & 1 & 0\\0 & 0 & 0 & 0 & 1 \\1 & 0 & 0 & 0 & 0},\;P = \diag\{a,a,1,1\}
$$
for any 
$$
a \geq {1 \over 4\theta^{m}_4}\left[\theta_2^{M} +{ (\theta_1^{M})^2 \over \theta_2^{m}} \right].
$$
\end{lemma}

\begin{proof}
From \eqref{psic} compute the mapping 
$$
\calw(\eta)=\cals(\cald^I(\eta))=\col(\eta_2 \eta_3 + \eta_1 \eta_4, \eta_1 \eta_3, \eta_2 \eta_3, \eta_3, \eta_4).
$$
and the matrix
$$
C =\begmat{I_4 & | & 0_{4 \times 1}}T=\begmat{0 & 1 & 0 & 0 & 0\\ 0 & 0 & 1 & 0 & 0\\0 & 0 & 0 & 1 & 0\\0 & 0 & 0 & 0 & 1}.
$$
Hence the ``good" mapping is $\calg(\eta)=\col(\calw_2(\eta),\calw_3(\eta),\calw_4(\eta),\calw_5(\eta))$, whose Jacobian yields
$$
\nabla \calg(\eta) = \left[ \begin{array}{cccc}  \eta_3 &  0& 0 &  0  \\ 0&  \eta_3 & 0 & 0\\ \eta_1 & \eta_2 & 1 & 0 \\ 0 &0 & 0 &1 \end{array}\right].
$$
Since the real part of the eigenvalues of this matrix are positive and its a Metzler matrix it admits a diagonal matrix $P$ such that \eqref{demcon2} holds \cite{BERPLE}. Computing the matrix 
$$
P \nabla \calg(\eta)+  [\nabla \calg(\eta)]^\top P=\left[ \begin{array}{cccc} 2 a\eta_3 &  0& \eta_1 &  0  \\ 0& 2a \eta_3 & \eta_2 & 0\\ \eta_1 & \eta_2 & 2 & 0 \\ 0 &0 & 0 &2 \end{array}\right],
$$
we see that it is positive definite if and only if its Schur complement of the $(2,2)$ block, given as,
$$
{1 \over 4}\begmat{4a  \eta_3 - \eta^2_1 & - \eta_1  \eta_2 \\ -\eta_1 \eta_2 & 4a  \eta_3 - \eta^2_2},
$$
is positive definite.  This, in its turn, is true if and only if
$$
a>{1 \over 4\eta_3}(\eta^2_1+\eta^2_2).
$$
The proof is completed bounding the right hand side from above, replacing $\eta$ by $\theta$ and using the bounds \eqref{bouthe}. 
\end{proof}
\subsubsection{Adaptive Slotine-Li control of the $2$-DOF robot manipulator}
\lab{subsubsec525}
In this subsubsection we present in detail the adaptive controller of Proposition \ref{pro7} with the Slotine-Li scheme for the 2-DOF robot manipulator. We show simulation results comparing the proposed scheme with the classical one relying on overparameterization.
 
To derive the NPRE \eqref{MSNPRE} we invoke \eqref{parmu} and \eqref{MU} and define
\begalis{
m_1 & :=\begmat{ 1 & 0 \\ 0& 0},\; m_2(q_2):=\cos(q_2)\begmat{ 2  & 1 \\ 1& 0},\; m_3:=\begmat{ 0 & 1 \\ 1& 1} \\
{\mathbb{U}}_1(q)&:=g[1+\sin(q_1+q_2)],\; {\mathbb{U}}_2(q_1):=g[1+\sin(q_1)].
}
Thus, the regressor matrix (\ref{MSNPRE}) takes the form 
$$
\Omega(q,\dot q)=H(p) \left[ \begin{array}{ccccc}  p \dot q_1& p \cos(q_2) (2\dot q_1 + \dot q_2) & p\dot q_2 &  g \cos(q_1+q_2) & g\cos(q_1)\\ 0 & p \cos(q_2) \dot q_1+\sin(q_2) (\dot q^2_1 + \dot q_1 \dot q_2)& p(\dot q_1 + \dot q_2) & g \cos(q_1+q_2)  & 0
      \end{array}\right].
$$

The known parameter version of the Slotine-Li controller \eqref{slolicon} may be parameterized as
$$
\beta(q,\dot q,\theta,t) := W(q,\dot q,t) \cals(\theta) + K_1s,
$$
with the matrix 
$$
W(q,\dot q,t):= \left[ \begin{array}{ccccc} \ddot q_{r1}&  \cos(q_2)(2\ddot q_{r1} + \ddot q_{r2}) -\sin(q_2)(\dot q_2\dot q_{r1} + (\dot q_1 + \dot q_2)\dot q_{r2}) & \ddot q_{r2} &  g \cos(q_1+q_2) & g\cos(q_1)\\ 0 & \cos(q_2) \ddot q_{r1} +\sin(q_2)\dot q_1\dot q_{r1} & \ddot q_{r1} + \ddot q_{r2} & g \cos(q_1+q_2)  & 0
      \end{array}\right],
$$
where $\dot q_r$ and $s$ are defined in (\ref{SLerrors}). In its standard version \cite{SLOLItac}, to get a linear parametrization, the adaptive implementation is obtained estimating the vector $\cals$, yielding
$$
\beta(q,\dot q,\hat S,t) := W(q,\dot q, \dot q_r, \ddot q_r) \hat S + K_1s.
$$
The parameter estimator is given as
$$
\dot { \hat S} :=-\Gamma W^\top(q,\dot q,t)s,
$$ 
that, as shown in \cite{SPOetal}, yields a globally stable closed-loop system and ensures global tracking of the desired references. 

In the proposed approach we estimate directly $\theta$, that is, the adaptive control is
$$
\beta(q,\dot q,\hat \theta,t) := W(q,\dot q, \dot q_r, \ddot q_r) \cals(\hat \theta) + K_1s,
$$
with the parameter estimator \eqref{kfil}, \eqref{esteta}, \eqref{yel} and \eqref{omeel}, combined with $\hat \theta=\cald^I(\hat \eta)$, where the mapping $\cald^I(\cdot)$ is given in (\ref{DI2gdl}). 

Now, we present some simulations comparing both approaches. For both controllers the gains are set as $K_1=3I_2$, $K_2=I_2$ and $\Gamma=5 I_2$. For the DREM-based controller the filter \eqref{fil} is implemented with $\lambda=2$ in Proposition \ref{pro1} and  $\lambda=1$ in Proposition \ref{pro6}, both filters with zero initial conditions.  The unknown parameters are set as $\theta_1=0.7$m; $\theta_2=0.8$m; $\theta_3=1.5$kg; and $\theta_4=0.5$kg. The initial velocities are set to zero and the initial positions are $q(0)=[0.2\pi; \;0.3\pi]$rad. The initial estimates are $\hat\theta_i(0)=0.01$ and $\hat \cals_i(0)=0.01$. The desired trajectory is 
$$
q_\star(t)=\col( 0.4\pi\sin(2t) + 0.2\pi,\; 0.3\pi\cos(t) +0.3\pi).
$$

Figure~\ref{fig:2gdl} shows the results of the simulations of the DREM-based and the standard schemes, from which we can observe that the trajectory tracking and the parameter estimation capabilities of our proposal clearly outperforms those of the classical adaptive controller. In this figure it can be also seen that consistent parameter estimation is quickly achieved. However, as indicated in Remark \ref{rem6}, the individual estimation errors $\tilde \theta_i$ are not monotonically decreasing. 

In Figure~\ref{fig:2gdlN} we change the initial conditions of the estimated parameters. From this figure we conclude that these initial conditions strongly affect the excitation of the system, encrypted in the signal $\Delta$ in \eqref{esteta}. Notice that, although there is a ``pattern" in the behavior of $\Delta^2$---as a function of the initial conditions---this is hard to predict. A similar ``sensitivity" to variations in the estimator and controller gains was observed, rendering difficult their tuning to achieve a satisfactory transient performance. The figure also shows that the norm of the estimation error $\tilde \eta$ is monotonically decreasing---as indicated in Proposition \ref{pro3}.

\begin{figure}[H]
\centering
  \includegraphics[width=0.8\textwidth]{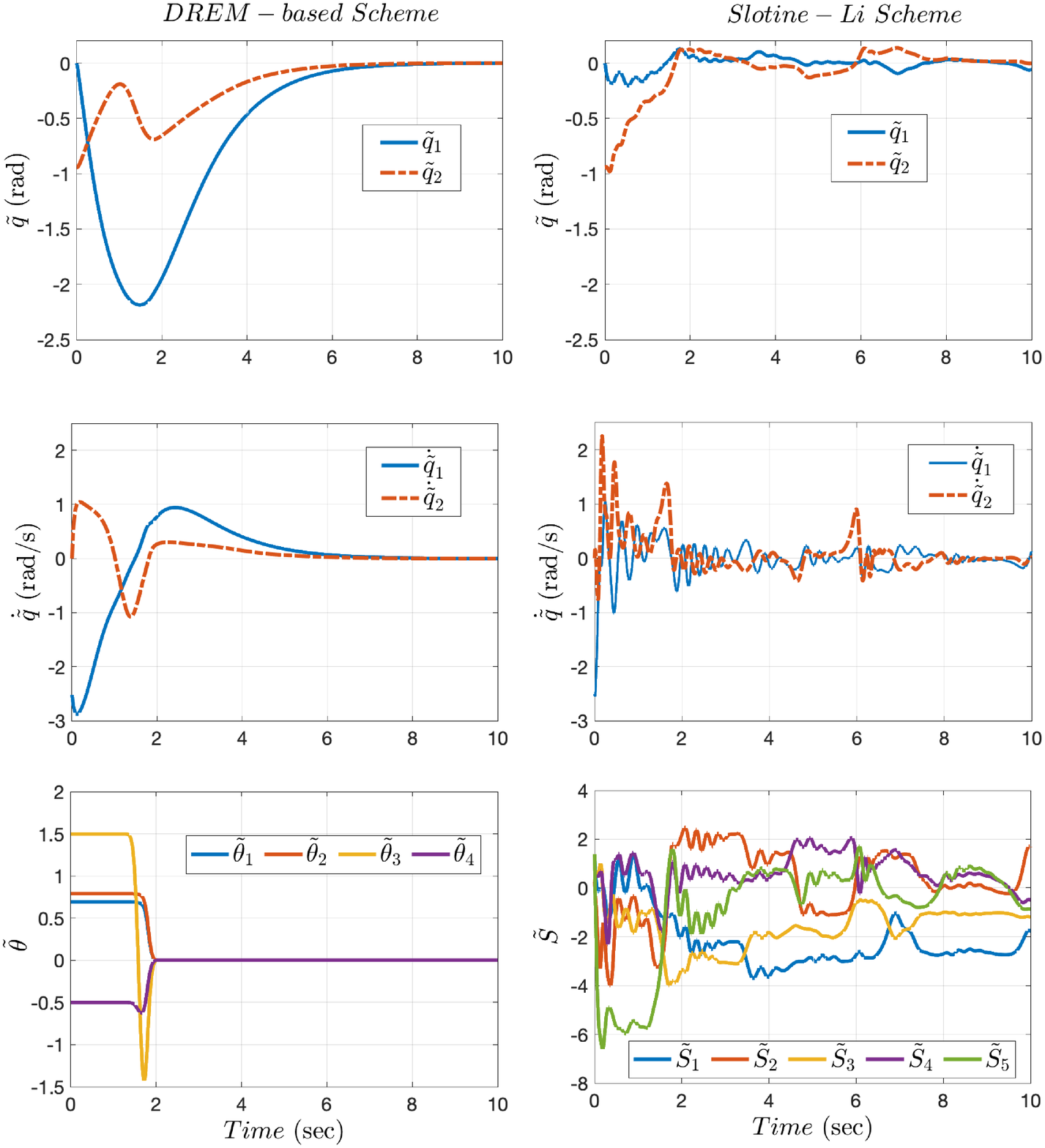}  
  \caption{Simulation results for the DREM-based adaptive scheme (left column) and the classical adaptive Slotine-Li (right column).}
  \label{fig:2gdl}
\end{figure}

\begin{figure}[H]
\centering
  \includegraphics[width=0.8\textwidth]{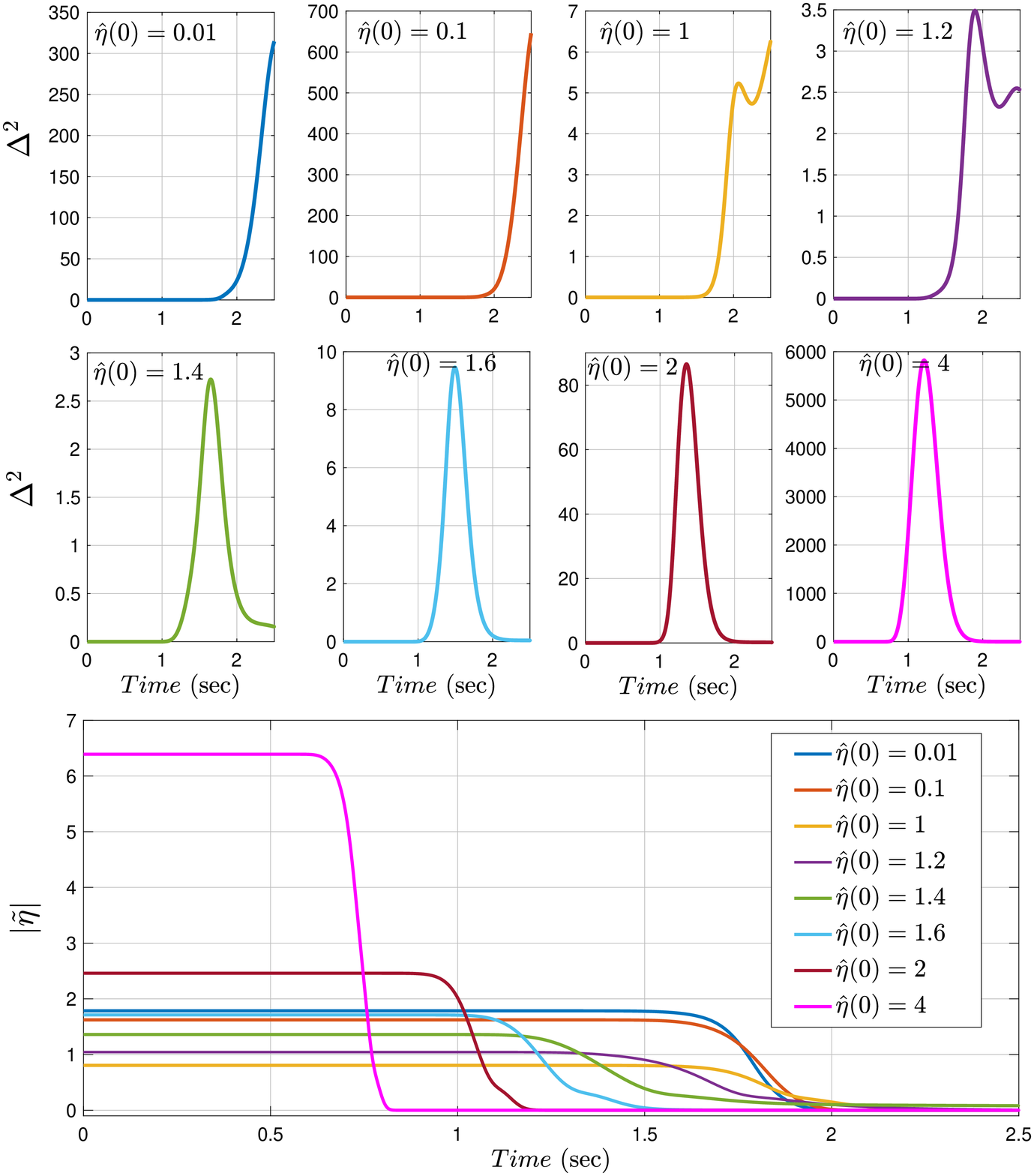}  
  \caption{``Measure of excitation" ($\Delta^2$) and norm of the parameter estimation error ($|\tilde \eta|$) for different initial conditions of the estimated parameters.}
  \label{fig:2gdlN}
\end{figure}
%
\section{Application to Nonlinearly Parameterized DT Systems}
\lab{sec6}
In this section we show how the proposed DREM-based parameter estimator can be applied to the problems of identification of a nonlinearly parameterized DT plant and to solve the direct and indirect versions of APPC. 
\subsection{Identification of a solar-heated house model}
\lab{subsec61}
%
In \cite[Example 1.1]{LJU} the problem of identification of the parameters of a solar-heated house model is discussed. The system operates in such a way that a sun heats the air in the solar panel, this air is then fanned into the heat storage. The stored energy can later be transferred to the house. The model of how the storage temperature $y_p(k) $ is affected by the fan control $u(k)$ and solar intensity $I(k)$ is given in  \cite[Example 5.1]{LJU} as
\begali{
\label{exampsys}
y_p(k)  =& (1-\theta_2)y_p(k-1) + (1-\theta_4)\frac{y_p(k-1)u(k-1)}{u(k-2)} + (\theta_4-1)(1+\theta_2)\frac{y_p(k-2)u(k-1)}{u(k-2)} + \\
&\theta_1\theta_3 u(k-1)I(k-2) - \theta_1 u(k-1)y_p(k-1) + \theta_1 (1+\theta_2)u(k-1)y_p(k-2),\nonumber
}
where $y_p(k) , u(k), I(k)$ are measurable scalar variables and $\theta:=\col(\theta_1,\dots,\theta_4)$ is a vector of constant, unknown, {\em physical} parameters of the system to be estimated. See  \cite[Example 5.1]{LJU} for an explanation of the physical meaning of the parameters $\theta$. 

Defining
\begequ
\lab{omesdt}
\Omega^\top(k):=\begmat{y_p(k-1) \\ \frac{y_p(k-1)u(k-1)}{u(k-2)} \\ \frac{y_p(k-2)u(k-1)}{u(k-2)} \\ u(k-1)I(k-2) \\ u(k-1)y_p(k-1) \\ u(k-1)y_p(k-2)}, \; 
\cals(\theta) := \begmat{1-\theta_2 \\ 1-\theta_4 \\ (\theta_4-1)(1+\theta_2) \\ \theta_1\theta_3 \\ - \theta_1 \\ \theta_1 (1+\theta_2)}.
\endequ
the model \eqref{exampsys} can be rewritten as the NLPRE \eqref{nlpre} that, as shown below, verifies the required assumptions for the direct estimation of $\theta$.

\begin{lemma}\em
\lab{lem3}
The NLPRE \eqref{nlpre}, \eqref{omesdt} verifies Assumptions \ref{ass1} and \ref{ass2} with the mapping $D : \rea^4 \rightarrow \rea^4$
$$
\eta = \cald(\theta) = \col(- \theta_1, 1-\theta_2, \theta_1\theta_3, 1-\theta_4),
$$
with right inverse $\cald^I : \rea^4 \rightarrow \rea^4$
\begequ
\theta = \cald^I(\theta) = \col(-\eta_1, 1 - \eta_2, - \frac{\eta_3}{\eta_1}, 1 - \eta_4),
\label{sim1_inv}
\endequ
the matrices 
$$
T=\begmat{0 & 0 & 0 & 0 & 1 & 0 \\ 1 & 0 & 0 & 0 & 0 & 0 \\ 0 & 0 & 0 & 1 & 0 & 0 \\ 0 & 1 & 0 & 0 & 0 & 0 \\ 0 & 0 & 0 & 0 & 0 & 1 \\ 0 & 0 & 1 & 0 & 0 & 0},\; P=I_4,
$$
and the constants $\nu=1$, $\rho=2$ and $\kappa=3$.
\end{lemma}
\begin{proof}
Compute the mapping
$$
\calw(\eta):=\cals(\cald^I(\eta))=\col(\eta_2, \eta_4, \eta_4(\eta_2-2), \eta_3, \eta_1, \eta_1(\eta_2 - 2)),
$$
and the matrix
$$
C:=\left[ I_4 \,|\, 0_{4\times2} \right] T = \begmat{0 & 0 & 0 & 0 & 1 & 0 \\ 1 & 0 & 0 & 0 & 0 & 0 \\ 0 & 0 & 0 & 1 & 0 & 0 \\ 0 & 1 & 0 & 0 & 0 & 0}.
$$
Hence the ``good'' mapping is 
$$
\calg(\eta)=\col(\calw_5(\eta), \calw_1(\eta), \calw_4(\eta), \calw_2(\eta))=\eta,
$$
with obvious Jacobian $\nabla \calg(\eta)=I_4$, which clearly satisfies \eqref{demcon2} and \eqref{lipcon} with the constants $\nu=1$ and $\rho=2$, respectively. 
\end{proof}

In \cite[Fig. 1.4]{LJU} an experimental record of the signals $y_p(k) , u(k), I(k)$ over a 16-hour period, sampled every 10 minutes, is given. The solar intensity $I(k)$ changes periodically with decaying form from the beginning till the end of the day, while the fan control $u(k)$ acts like a pulse signal with only two possible values. For simulation purposes a similar behavior of these signals was recreated and is presented in Fig.~\ref{fig_sim1_ctrl}. 

\begin{figure}[H]
\centering
\includegraphics[width=\textwidth]{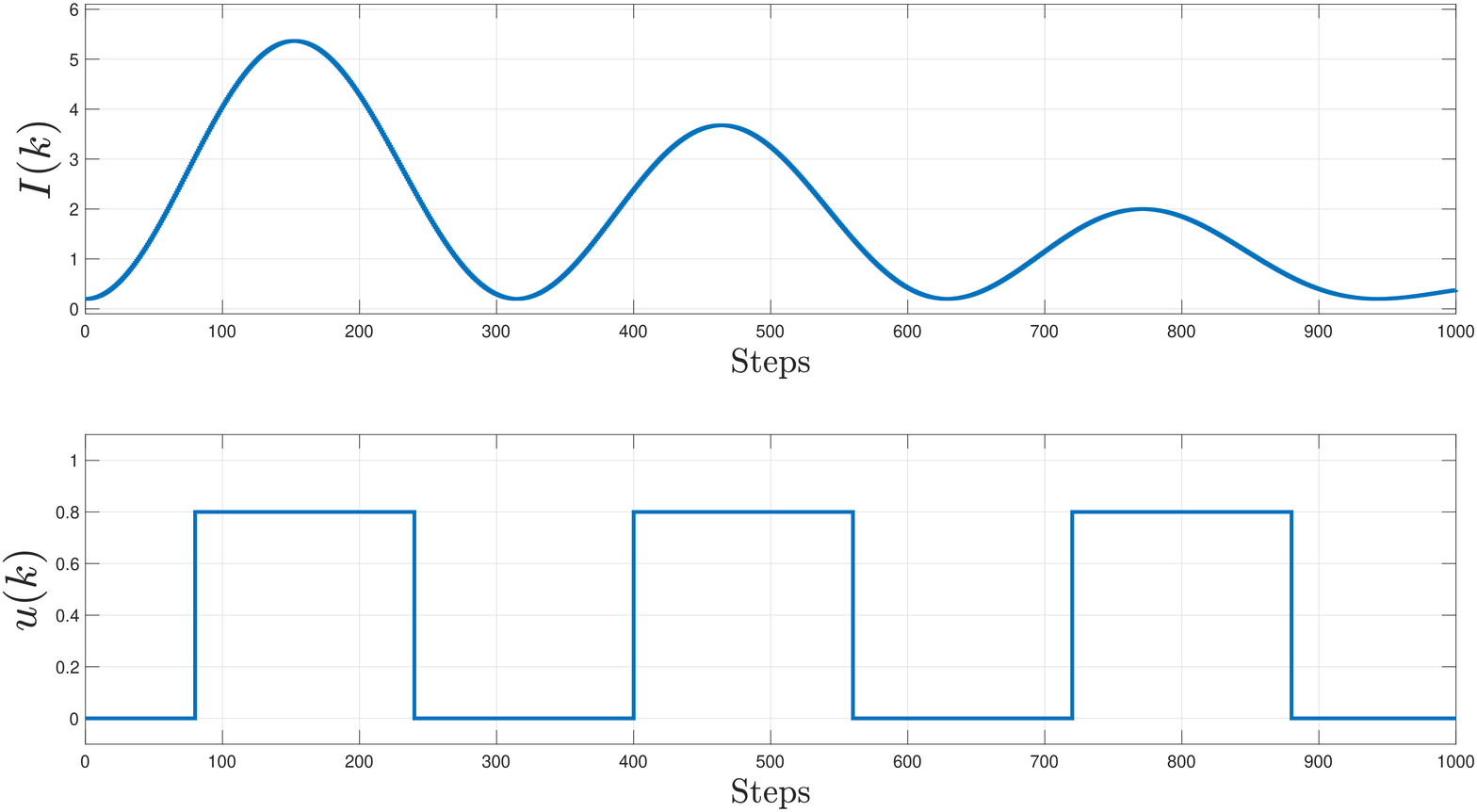}
\caption{External signals $I(k)$ and $u(k)$ used for the simulation of the solar-heated house model}
\label{fig_sim1_ctrl}
\end{figure}

The DREM-based estimator of Propositions \ref{pro2} and \ref{pro4}, with the filter pole at $\alpha=0.9$ and the adaptation gain $\gamma=1$, was simulated. To comply with \eqref{gaicon} we fixed $\kappa=3$. The value of the system parameters used in the simulations was $\theta_i=0.5, i=1,\dots,4$, and the estimator initial conditions were chosen as $\hat\eta_i(0)=\eta_i-0.5, i=1,\dots,4$.\footnote{It was observed that the behavior of the estimator remains unchanged for other values of these parameters and other initial conditions. } The transient behavior of the parameter estimation errors $\tilde\eta_i(k)$ are presented in Fig.~\ref{fig_sim1_dr}. The plot shows that convergence is achieved after the second pulse in $u(k)$. Also, although not predicted by he theory we observe a monotonic behavior of {\em each} error signal. Using the inverse transformation \eqref{sim1_inv} it is possible to calculate the estimations of model parameters $\hat\theta_i(k)$ which are shown in Fig.~\ref{fig_sim1_dr} as well.

\begin{figure}[H]
\centering
\includegraphics[width=\textwidth]{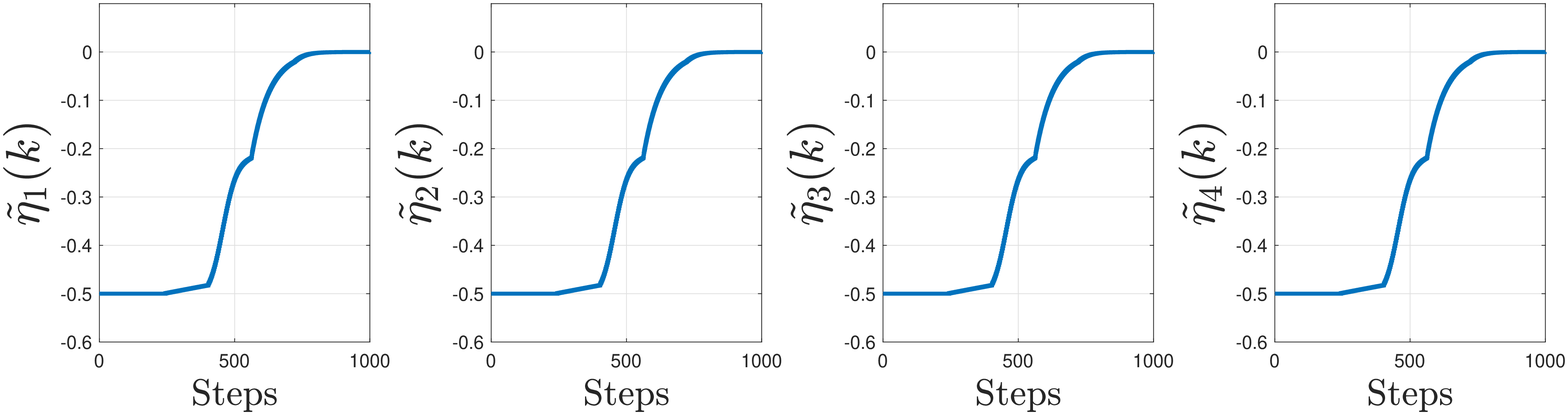}\\
\includegraphics[width=\textwidth]{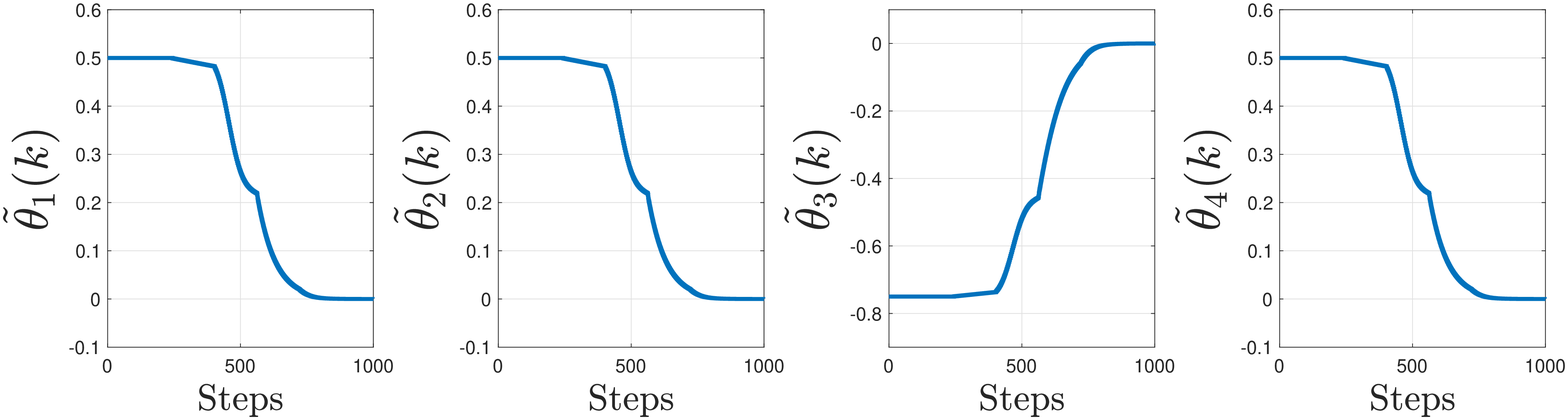}
\caption{Transient behaviour of the estimation errors  $\tilde \eta_i(k)$ and estimation errors $\tilde \theta_i(k)$ of the solar-heated house model using DREM-based estimator}
\label{fig_sim1_dr}
\end{figure}

In \cite{LJU} it is proposed to overparameterize the NPRE to obtain a linear regression. As indicated there, the price that is paid is that the value of the physical parameters $\theta$---which might be of interest in some applications---{\em cannot be recovered} from the knowledge of $\cals(\theta)$. Clearly, this is not the case for the proposed scheme since $\theta$ can be calculated with the inverse transformation \eqref{sim1_inv}. In any case, for performance comparison purposes a simulation was carried out with the overparameterized model \eqref{omesdt} using the standard gradient estimator
$$
\hat{\cals}(k) = \hat{\cals}(k-1) + {\Omega^\top(k) \over \gamma + \Omega(k) \Omega^\top(k)} [y(k) - \Omega(k) \hat{\cals}(k-1)],
$$
with $\gamma=1$. Simulation results are shown in Fig.~\ref{fig_sim1_gr}. As seen from the plots, the parameters converge faster than the DREM estimator, but they converge to {\em wrong} values.

\begin{figure}[H]
\centering
\includegraphics[width=\textwidth]{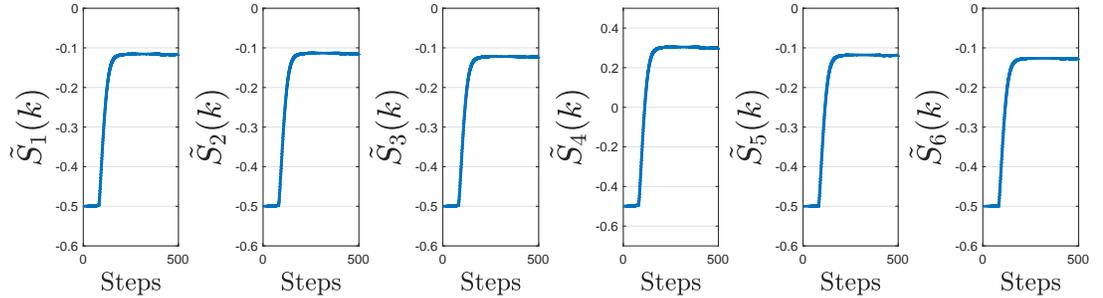}\\
\caption{Transient behaviour of the estimation errors  $\tilde S_i(k)$ of the overparameterized solar-heated house model using gradient estimator }
\label{fig_sim1_gr}
\end{figure}
%
\subsection{Adaptive Pole Placement Control of LTI Systems}
\lab{subsec62}
%
We are interested in this subsection in the problem of APPC of LTI DT system represented by it pulse transfer function
\begin{align}
	A(q^{-1})y_p(k) =B(q^{-1})u(k),	
\label{dtsys}
\end{align}
where the polynomials
$$
A(q^{-1})=1 +a_1q^{-1} + \cdots +a_{n_A}q^{-n_A},\;B(q^{-1})=b_0 +b_1q^{-1} + \cdots +b_{n_B}q^{-n_B},
$$
are coprime, with a known {\em upperbound} on their order, say $v$, but with {\em unknown} coefficients $a_i,b_i$. The pole-placement problem consists of designing a controller 
 \begin{equation}
 \label{dtcon}
L(q^{-1}) u(k) + P(q^{-1})y_p(k) =r(k)
 \end{equation}
such that the closed-loop system takes the form
$$
y_p(k) ={B(q^{-1}) \over A_m(q^{-1})}r(k),
$$
where $r(k)$ is a bounded external signal and $A_m(q^{-1})=1 +a^m_1q^{-1} + \cdots +a^m_{n_{A_m}}q^{-n_{A_m}},$ is a desired closed-loop polynomial whose roots are inside the unit circle. That is, the controller relocates the poles of the system in a desired position but preserves the open-loop zeros. For a lucid exposition of this problem see \cite[Section 5.3]{GOOSIN} and \cite{PYRetal} for a review of the recent literature.
\subsubsection{Obstacles for the adaptive implementation}
\lab{subsubsec621}
%
Computing \eqref{dtsys} in closed-loop with \eqref{dtcon} we get 
 \begin{equation}
 y_p(k)  =\frac{B(q^{-1})}{A(q^{-1})L(q^{-1}) +B(q^{-1})P(q^{-1})} r(k).
 \end{equation}
Hence, to achieve the objective, we need to verify the Bezout equation  
\begequ
\lab{sylequ}
 A(q^{-1})L(q^{-1})+{ B}(q^{-1})P(q^{-1})=A_m(q^{-1}).
\endequ
As is well-known \cite[Theorem 5.3.1]{GOOSIN}, selecting $n_{A_m}:=2v-1$, there exists unique polynomials $L(q^{-1})$ and $P(q^{-1})$, both of order $(v-1)$, solutions of \eqref{sylequ}. Indeed, it is possible to show that \eqref{sylequ} admits a matrix representation 
\begequ
\lab{matsylequ}
S(a_i,b_i) \eta =\col(a^m_0,a^m_1,\dots,a^m_{2v-1}),
\endequ
where 
\begequ
\lab{eta}
\eta:= \col(l_0,l_1,\dots,l_{v-1},p_0,p_1,\dots,p_{v-1})
\endequ 
and $S(a_i,b_i) \in \rea^{2v \times 2v}$---called the Sylvester  matrix---is linearly dependent on the coefficients  $a_i,b_i$, and is {\em full rank} if and only if $A(q^{-1})$ and $B(q^{-1})$ are coprime.

It is well-known that the adaptive version of the previous controller, called APPC, suffers from serious drawbacks \cite{GOOSIN,PYRetal}. In its {\em indirect} version---that is when we estimate the parameters of the plant $a_i,b_i$ and then compute from them, via the solution of \eqref{matsylequ}, the parameters of the controller $l_i,p_i$---the problem is that the Sylvester matrix with the estimated parameters $\hat a_i(k),\hat b_i(k)$ may loose rank during the transient behavior. Although this phenomenon can be avoided adding parameter projections, the prior knowledge required to implement this efficiently is never available in practice and relies on the availability of PE, see \cite[Section 1]{PYRetal}.  

On the other hand, in its {\em direct} version the estimation of the {\em controller} parameters involves a NPRE. Indeed, applying \eqref{sylequ} to the output of the plant $y_p(k) $ we get
\begali{
\nonumber
& L(q^{-1})A(q^{-1})y_p(k) +P(q^{-1})B(q^{-1})y_p(k)   =A_m(q^{-1})y_p(k) \\
\nonumber
& \Leftrightarrow L(q^{-1})B(q^{-1})u(k)+P(q^{-1})B(q^{-1})y_p(k)   =A_m(q^{-1})y_p(k) \\
\lab{pprep}
& \Leftrightarrow B(q^{-1})[L(q^{-1})u(k)+P(q^{-1})y_p(k) ]  =A_m(q^{-1})y_p(k) ,
}
where we invoked \eqref{dtsys} to get the second equation. The known parameter version of the direct pole-placement controller may be written in the LRE form
$$
u(k) + \eta^\top  \psi(k)=r(k)
$$
where we have used the fact that $L(q^{-1})$ is monic and defined
$$
\psi(k):=\col(y_p(k) ,\dots,y_p(k-v+1),u(k-1),\dots,u(k-v+1)) \in \rea^{2v-1},
$$
with $\eta$, as defined in \eqref{eta}, contains the unknown coefficients of the polynomials $L(q^{-1})$ and $P(q^{-1})$. A direct adaptive implementation of this controller takes then the form
$$
u(k) +\hat \eta^\top (k) \psi(k)=r(k),
$$
where $\hat \eta(k)$ denotes the estimates of $\eta$. The difficulty of designing an estimator for the controller parameters $\eta$ is due to the fact that, in terms of $\eta$, \eqref{pprep} defines a parameterization of  the form 
\begequ
\lab{bthepeta}
B(q^{-1})[u(k)+\eta^\top  \psi(k)]=A_m(q^{-1})y_p(k)=:y_p(k),
\endequ
which is {\em bilinear} because the polynomial $B(q^{-1})$ is {\em unknown}. 

In the next two subsubsections we show that using the results reported in the paper it is possible to overcome the two obstacles mentioned above. To simplify the presentation we illustrate this fact with simple representative  examples, that can be easily extended to the general case. 
 
\subsubsection{DREM-based indirect APPC}
\lab{subsubsec622}
%
Consider the LTI DT system
\begin{equation}
\label{exLTI}
 y_p(k+1) + \theta y_p(k)= u(k)+ \theta^3 u(k-1),
\end{equation}
where, to ensure the coprimeness assumption, $\theta \neq \pm 1$. Fixing a dead-beat objective, {\em e.g.}, $A_m(q^{-1})=1$, and selecting $L(q^{-1})= l_0 + l_1 q^{-1}$ and $P(q^{-1}) = p_0 + p_1 q^{-1}$ the Bezout equation \eqref{sylequ} takes the form
 \begin{equation}
 (1+\theta q^{-1}) (l_0 + l_1 q^{-1}) + q^{-1} (1 + \theta^3 q^{-1}) (p_0 + p_1 q^{-1})=1.
 \end{equation}
The latter can be rewritten as 
 \begin{equation}
 \label{sysa}
 \left[ \begin{array}{ccccc}  1 & 0 & 0 & 0\\ \theta & 1& 1 &  0 \\
 0 & \theta & \theta^3 & 1 \\ 0 & 0 & 0 & \theta^3
      \end{array}\right]  \begmat{l_0 \\ l_1 \\ p_0  \\ p_1}= \begmat{1 \\ 0  \\ 0  \\ 0},
 \end{equation}
whose solution is  $l_0 =1$, $p_1 =0$  and 
 \begin{equation}
 \label{sysa0}
 \left[ \begin{array}{cc}    1& 1  \\
 \theta & \theta^3
      \end{array}\right] \begmat{l_1  \\ p_0 }  =\begmat{- \theta  \\ 0 },
       \end{equation}
which corresponds to 
$$
\begmat{l_1  \\ p_0 }=  \frac{1}{\theta^3 - \theta}\begmat{-\theta^4  \\ \theta^2 }.
$$
Hence,  the known-parameter controller \eqref{dtcon} takes the form
\begin{equation}
\label{conLTI}
u(k)=-\frac{1}{\theta^3-\theta} \left[ \theta^2 y_p(k)-\theta^4 u(k-1) \right]+r(k)
\end{equation}
and yields the desired closed-loop system 
$$
y_p(k)=q^{-1}(1+\theta^3 q^{-1}) r(k).
$$

Obviously, the system admits an NPRE  of the form \eqref{nlpre} with
 \begin{equation}
\label{yoms}
y_p(k):=  y_p(k)-u(k-1),\;  \Omega(k):=  \begmat{-y_p(k-1) \\ u(k-2)}, \hspace{3mm}  \cals(\theta):=  \begmat{\theta \\ \theta^3}.
 \end{equation}
If we overparametrize the NPRE and estimate the vector $\cals \in \rea^2$ the controller parameters are computed from 
 \begin{equation}
 \label{sysa1}
 \left[ \begin{array}{cc}    1& 1  \\
 {\hat \cals}_1(k) & {\hat \cals}_2(k)
      \end{array}\right] \begmat{\hat l_1(k)  \\ \hat p_0(k) }  =\begmat{- {\hat \cals}_1(k)  \\ 0 },
       \end{equation}
which yields the adaptive controller 
\begin{equation}
u(k)=- \frac{1}{\hat {\cals}_2(k)-\hat {\cals}_1(k)} \left [ {\hat \cals_1}^2(k)y_p(k) -  \hat {\cals}_1(k) {\hat \cals}_2(k) u(k-1)  \right] +r(k).
\end{equation} 
Clearly, the controller computation has a singularity on the line ${\hat \cals}_1(k)={\hat \cals}_2(k)$. On the other hand, if we estimate $\theta$, the adaptive version of \eqref{conLTI} has a singularity only at the points $\hat \theta(k)=\pm 1$. 

The simulation scenario was a system with {\em changing} parameters
$$
\theta=\;\left\{
\begarr{cc}
{1 \over 2} & 0 \leq t < 5  \\
 -{1 \over 2} & 5 \leq t.
\endarr
\right.
$$
The external signal $r(k)$ is a sinusoidal function. The initial conditions of the estimators were taken as $\hat \theta(0)=\hat \cals_1(0)={1 \over 2}$ and  $\hat \cals_2(0)={1 \over 8}$. For $ 0 \leq t < 5$ we have $\cals_1(\theta)<\cals_2(\theta)$ and  for $ t \geq 5$ we have $\cals_1(\theta) > \cals_2(\theta)$. Therefore, if the estimates $\hat S(k)$ converge they have to {\em cross through singularity}.  On the other hand, the DREM-based scheme shouldn't leave the singularity-free region $\theta \in (-1,1)$ because of the monotonicity property. 

The simulation results for the DREM-based estimation of $\theta$ with $\gamma=1$ and $\kappa=2$ are presented in Fig.~\ref{fig_sim2_dr}. As seen from the figure, the controller parameter error converges to zero and the estimated parameter $\hat\theta(k)$ does not leave the singularity-free region $\theta \in (-1,1)$. As expected, the tracking error $e(k) := y_p(k)-B(q^{-1})r(k)$ also converges to zero in the closed-loop system.

\begin{figure}[H]
\centering
\includegraphics[width=\textwidth]{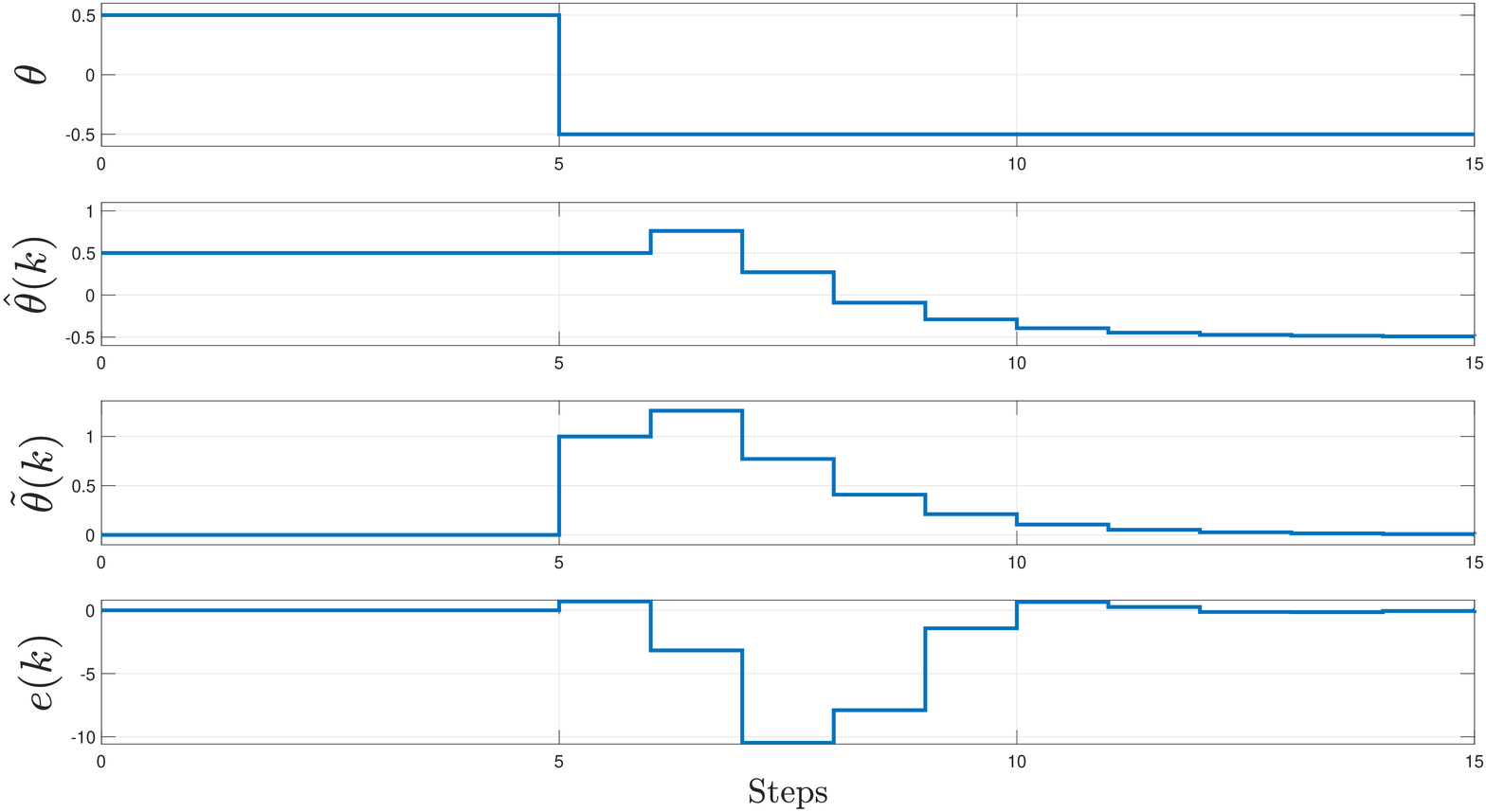}\\
\caption{Transient behaviour of the systems switched parameter $\theta$, its estimate $\hat\theta(k)$, the estimation error  $\tilde\theta(k)$ and the tracking error $e(k)$ in the indirect APPC task using DREM-based estimator}
\label{fig_sim2_dr}
\end{figure}

For performance comparison a simulation was completed with the overparameterized model \eqref{yoms} using the standard gradient estimator
$$
\hat{\cals}(k) = \hat{\cals}(k-1) + {\Omega^\top(k) \over \gamma + \Omega(k) \Omega^\top(k)} [y(k) - \Omega(k) \hat{\cals}(k-1)],
$$
with $\gamma=1$ and the adaptive controller \eqref{sysa1}. The simulation results are presented in Figs.~\ref{fig_sim2_grs} and ~\ref{fig_sim2_gr}. As seen from  Fig.~\ref{fig_sim2_grs} the estimated parameters cross the singularity line $\cals_1=\cals_2$. However, due to the DT nature of the equations, they ``jump" through it without inducing an unacceptable transient behavior in the control calculation---a coincidence that, of course, cannot be theoretically predicted. As seen from Figs.~\ref{fig_sim2_gr}, parameter and tracking error convergence is twice as slow as the one of the DREM estimator. 

\begin{figure}[H]
\centering
\includegraphics[width=\textwidth]{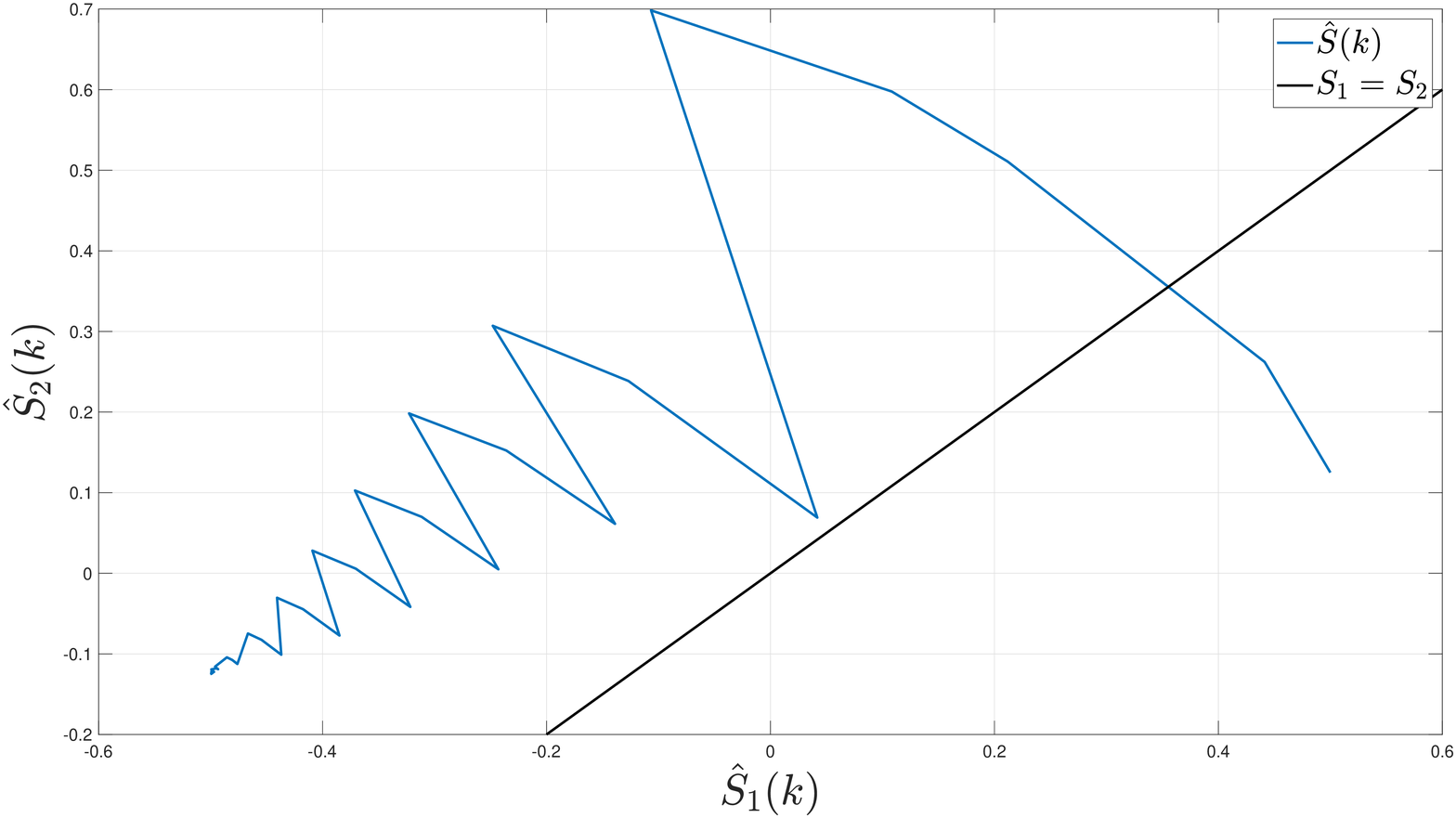}\\
\caption{Transient behaviour of the estimated parameters $\hat\cals_1(k)$ and $\hat\cals_2(k)$ and the singularity line $\cals_1=\cals_2$ in the plane $\cals_1-\cals_2$}
\label{fig_sim2_grs}
\end{figure}

\begin{figure}[H]
\centering
\includegraphics[width=\textwidth]{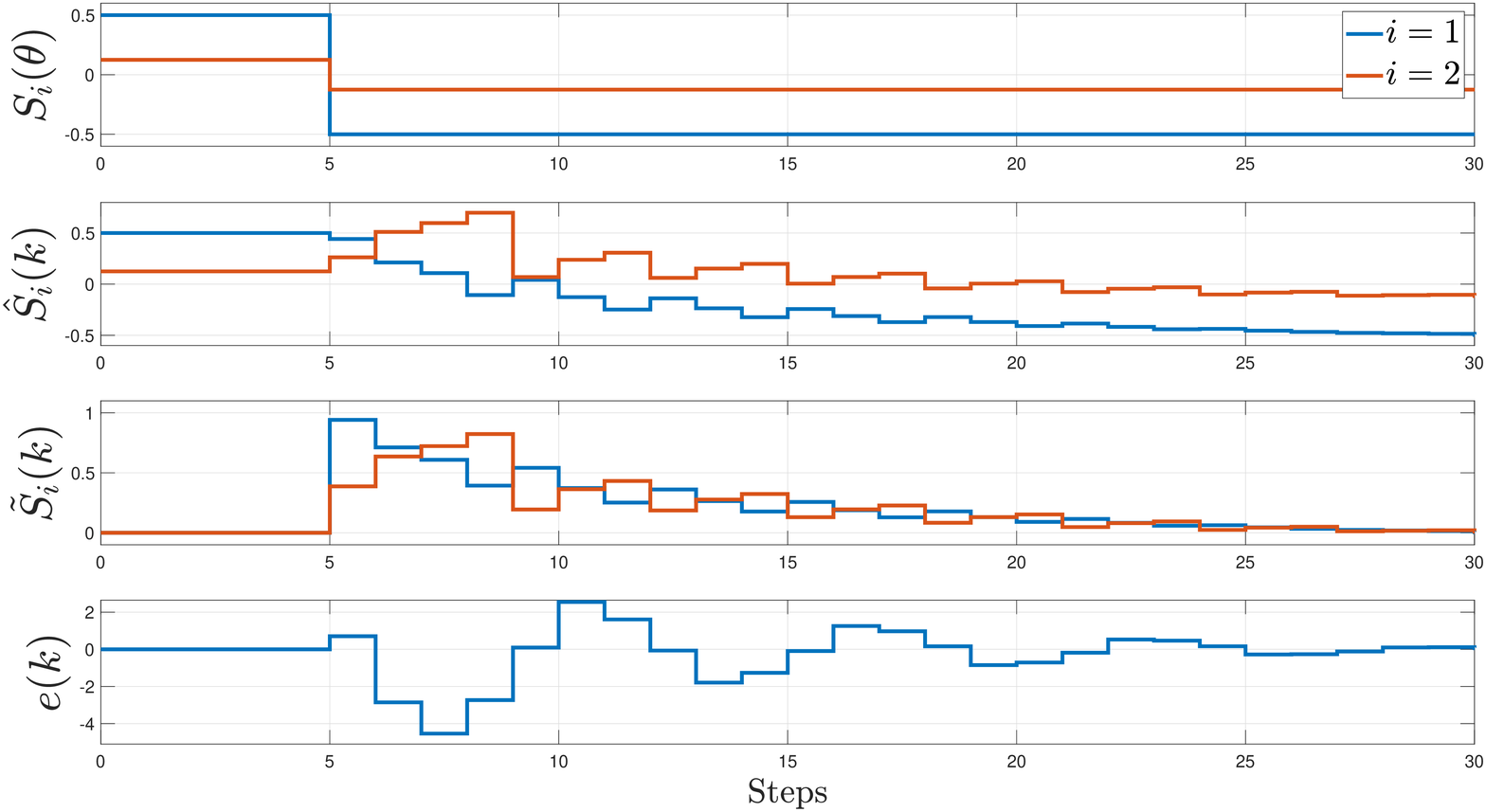}\\
\caption{Transient behaviour of the switching parameters of the system $\cals_i(\theta)$, their estimation $\hat\cals_i(k)$, the estimation error  $\tilde\cals_i(k)$ and the tracking error $e(k)$ in the overparameterized indirect APPC}
\label{fig_sim2_gr}
\end{figure}
\subsubsection{DREM-based direct APPC}
\lab{subsubsec623}
%
In this subsubsection we illustrate with a simple example how the DREM-based direct APPC avoids the bilinearity problem mentioned in Subsection \ref{subsec62}.  Towards this end, consider the DT system \eqref{dtsys} with
\begalis{
	A(q^{-1})&=1 +a_1q^{-1},\;B(q^{-1})=b_1q^{-1} +b_{2}q^{-2},
}
and choose a deadbeat control objective, that is, $A_m(q^{-1})=1$. Since $v=2$ the known parameter control law \eqref{dtcon} takes the form 
$$
(1 +l_1q^{-1})u(k)+(p_0 +p_{1}q^{-1})y_p(k) =r(k).
$$
Hence \eqref{pprep} becomes
$$
(b_1q^{-1} +b_{2}q^{-2})[(1 +l_1q^{-1})u(k)+(p_0 +p_{1}q^{-1})y_p(k) ]=y_p(k) .
$$
By solving equation \eqref{sylequ} it is easy to see that $p_1=0$, reducing the equation above to the form
$$
(b_1q^{-1} +b_{2}q^{-2})[(1 +l_1q^{-1})u(k)+p_0y_p(k) ]=y_p(k) .
$$
Some simple calculations show that the latter may be written in the $5$-dimensional LRE form 
\begequ
\lab{dtlre}
y_p(k)   = \Omega^\top (k) \cals(\theta),
\endequ
where
\begali{
\nonumber
\theta &:=\col(b_1,b_2,p_0,l_1)\\
\nonumber
\Omega(k)&:=\col(y_p(k-1),y_p(k-2),u(k-1),u(k-2),u(k-3))\\
\cals(\theta)&:=\col(\theta_1 \theta_3,\theta_2 \theta_3,\theta_1,\theta_1 \theta_4+\theta_2,\theta_2 \theta_4).
\lab{thephi}
}
The bijective mapping $D : \rea^4\rightarrow \rea^4$
$$
\eta = \cald(\theta) = \col(\theta_1, \theta_2\theta_3, \theta_1\theta_3, \theta_2\theta_4),
$$
with right inverse $\cald^I : \rea^4 \rightarrow \rea^4$
\begequ
\theta = \cald^I(\theta) = \col\Big(\eta_1, \frac{\eta_2\eta_1}{\eta_3}, \frac{\eta_3}{\eta_1}, \frac{\eta_3\eta_4}{\eta_2\eta_1}\Big),
\label{sim2_inv}
\endequ
verifies Assumption \ref{ass1} with,
$$
T:=\begmat{0 & 0 & 1 & 0 & 0 \\ 0 & 1 & 0 & 0 & 0 \\ 1 & 0 & 0 & 0 & 0 \\ 0 & 0 & 0 & 0 & 1 \\ 0 & 0 & 0 & 1 & 0},\; P=I_4.
$$
Indeed, computing the mapping
$$
\calw(\eta):=\cals(\cald^I(\eta))=\col\Big(\eta_3, \eta_2, \eta_1, \frac{\eta_3\eta_4}{\eta_2} - \frac{\eta_2\eta_1}{\eta_3}, \eta_4\Big),
$$
and the matrix
$$
C:=\left[ I_4 \,|\, 0_{4\times1} \right] T = \begmat{0 & 0 & 1 & 0 & 0 \\ 0 & 1 & 0 & 0 & 0 \\ 1 & 0 & 0 & 0 & 0 \\ 0 & 0 & 0 & 0 & 1 }
$$
Hence, we get the ``good'' mapping is 
$$
\calg(\eta)=\col(\calw_3(\eta), \calw_2(\eta), \calw_1(\eta), \calw_5(\eta))=\eta,
$$
whose Jacobian is $\nabla \calg(\eta)=I_4$, which clearly satisfies \eqref{demcon2} and \eqref{lipcon} with the constants $\nu=1$ and $\rho=2$, respectively..

%
%
%

\section{Conclusions}
\lab{sec7}
%
It has been shown that the DREM procedure can be used to estimate the parameters of a CT or DT NPRE of the form \eqref{nlpre}, provided the ``monotonizability" Assumption \ref{ass1} holds and some weak excitation conditions---encrypted in the scalar signal $\Delta$---are satisfied. The applicability of the method has been illustrated with several classical examples. 

We are currently pursuing the following research avenues. 
\begenu[{\bf R1}]
\item As indicated in Remark \ref{rem12} the highly attractive parameterization of EL systems proposed in \cite{SLOLIaut} seems to yield a non-identifiable NPRE. A rigorous proof of this claim is yet to be established.

\item Although the DREM estimator has a few tuning gains, {\em e.g.}, the filter constants ($\lambda$ for CT, and $\alpha$ for DT) and the adaptation gain $\gamma$, their impact on the transient behavior is hard to predict---see Subsubsection \ref{subsubsec525}. A more thorough analysis of the sensitivity of the design {\em vis-\`a-vis} these coefficients is yet to be derived.

\item Although avoiding overparameterization to handle NPRE seems, in principle, a sensible objective, it is not clear under which conditions this approach is really more convenient. Particularly considering that this is, until now, only applicable to ``monotonizable" NPRE. 

\item The verification of the conditions of Proposition \ref{pro1} is carried out in our examples via direct inspection. A deeper understanding of the underlying structural features of the mapping $\cals(\theta)$ under which this is possible would be highly desirable. It seems that such a study should appeal to principles of differential algebra.    
\endenu
%
\section*{Acknowledgment}
This paper is partly supported by the Ministry of Education and Science of Russian Federation (14.Z50.31.0031, goszadanie no. 8.8885.2017/8.9), NSFC (61473183, U1509211) and the Mexican CONACyT Basic Scientific Research grant CB-282807.



\end{document}